\documentclass{amsart}

\usepackage{diagrams}
\diagramstyle{nohug}
\usepackage{amssymb}

\usepackage{ifthen}

\newcounter{mylisti} \newcounter{mylistii}
\newcounter{nest}
\newcommand{\defaultlabel}{}

\newenvironment{mylist}[1]{%
  \addtocounter{nest}{1}
  \ifthenelse{\value{nest}=1}{%
    \renewcommand{\defaultlabel}{(\roman{mylisti})\hfill}}{%
    \renewcommand{\defaultlabel}{(\alph{mylistii})\hfill}}
  \begin{list}{\defaultlabel}{%
      \ifthenelse{\value{nest}=1}{\usecounter{mylisti}}{%
        \usecounter{mylistii}}
      
      \addtolength{\itemsep}{0.5ex}
      \settowidth{\labelwidth}{#1}
      \setlength{\leftmargin}{\labelwidth}
      \addtolength{\leftmargin}{\labelsep}}}{\addtocounter{nest}{-1}
\end{list}}

\newcommand{\Sb}{\ensuremath{\overline{S}}}

\newcommand{\be}{\ensuremath{\mathbb E}}
\newcommand{\bn}{\ensuremath{\mathbb N}}
\newcommand{\bp}{\ensuremath{\mathbb P}}

\newcommand{\br}{\ensuremath{\mathbb R}}

\newcommand{\cB}{\ensuremath{\mathcal B}}

\newcommand{\cF}{\ensuremath{\mathcal F}}
\newcommand{\cG}{\ensuremath{\mathcal G}}

\newcommand{\cJ}{\ensuremath{\mathcal J}}
\newcommand{\cK}{\ensuremath{\mathcal K}}
\newcommand{\cL}{\ensuremath{\mathcal L}}

\newcommand{\cX}{\ensuremath{\mathcal X}}

\newcommand{\cGb}{\ensuremath{\overline{\mathcal G}}}
\newcommand{\cJb}{\ensuremath{\overline{\mathcal J}}}

\newcommand{\ut}{\ensuremath{\tilde{u}}}

\newcommand{\At}{\ensuremath{\tilde{A}}}
\newcommand{\Bt}{\ensuremath{\tilde{B}}}
\newcommand{\Dt}{\ensuremath{\tilde{D}}}

\newcommand{\Qt}{\ensuremath{\tilde{Q}}}
\newcommand{\Tt}{\ensuremath{\tilde{T}}}

\newcommand{\pit}{\ensuremath{\tilde{\pi}}}

\newcommand{\abs}[1]{\lvert #1\rvert}
\newcommand{\bigabs}[1]{\big\lvert #1\big\rvert}
\newcommand{\Bigabs}[1]{\Big\lvert #1\Big\rvert}

\newcommand{\Ave}{\operatorname{Ave}}

\newcommand{\diag}{\operatorname{diag}}

\newcommand{\id}{\operatorname{Id}}

\newcommand{\bi}{\ensuremath{\boldsymbol{1}}}

\newcommand{\bigintp}[1]{\ensuremath{\big\lfloor #1\big\rfloor}}

\newcommand{\ip}[2]{\ensuremath{\langle #1,#2\rangle}}
\newcommand{\bigip}[2]{\ensuremath{\big\langle #1,#2\big\rangle}}

\newcommand{\join}{\vee}
\newcommand{\bigjoin}{\bigvee}
\newcommand{\meet}{\wedge}

\newcommand{\norm}[1]{\lVert #1\rVert}
\newcommand{\bignorm}[1]{\big\lVert #1\big\rVert}
\newcommand{\Bignorm}[1]{\Big\lVert #1\Big\rVert}

\newcommand{\prob}[1]{\ensuremath{\bp(#1)}}
\newcommand{\bigprob}[1]{\ensuremath{\bp \big(#1\big)}}
\newcommand{\Bigprob}[1]{\ensuremath{\bp \Big(#1\Big)}}

\newcommand{\restrict}{\ensuremath{\!\!\restriction}}

\newcommand{\sgn}{\ensuremath{\mathrm{sign}}}

\newcommand{\spn}{\ensuremath{\mathrm{span}}}

\newcommand{\co}{\mathrm{c}_0}

\newcommand{\me}{\mathrm{e}}

\newcommand{\V}{\forall \,}
\newcommand{\E}{\exists\,}

\newcommand{\vare}{\varepsilon}

\newcommand{\ds}{\displaystyle}

\newcommand{\ts}{\textstyle}
\newcommand{\phtm}[1]{\text{\makebox[0pt]{\phantom{$#1$}}}}
\newcommand{\eg}{\textit{e.g.,}\ }

\newcommand{\ie}{\textit{i.e.,}\ }
\newcommand{\Ie}{\textit{I.e.,}\ }
\newcommand{\cf}{\textit{c.f.}\ }

\newcommand{\etc}{\textit{etc.}\ }

\newcommand{\thespace}{\ensuremath{\big(\bigoplus
    _{n=1}^\infty\ell_1^n \big)_{\co}}}

\newcommand{\sur}{\ensuremath{^{\mathrm{(sur)}}}}
\newcommand{\inj}{\ensuremath{^{\mathrm{(inj)}}}}

\newtheorem{thm}{Theorem}[section]
\newtheorem{mainthm}{Theorem}

\newtheorem{lem}[thm]{Lemma}
\newtheorem{prop}[thm]{Proposition}

\newtheorem*{problem}{Problem}
\newtheorem*{question}{Question}

\theoremstyle{definition}
\newtheorem*{defn}{Definition}

\theoremstyle{remark}
\newtheorem*{rem}{Remark}

\title[Dichotomy theorems for random matrices and ideals]{Dichotomy
  theorems for random matrices and closed ideals of operators on
  $\big(\bigoplus _{n=1}^\infty\ell_1^n \big)_{\co}$}
\author{N.~J.~Laustsen, E.~Odell, Th.~Schlumprecht and A.~Zs\'ak}
\subjclass[2000]{47L10 (primary), 46B09, 46B42, 47L20, 46B45 (secondary)}
\thanks{The authors gratefully acknowledge the financial support
from the EPSRC, grant EP/F023537/1, and the NSF, grants DMS 0700126
and DMS 0856148. Commutative diagrams were drawn by Paul Taylor's
  package.}

\begin{document}

\begin{abstract}
We prove two dichotomy theorems about sequences of operators into
$L_1$ given by random matrices. In the second theorem we assume that
the entries of each random matrix form a sequence of independent,
symmetric random variables. Then the corresponding sequence of
operators either uniformly factor the identity operators on $\ell_1^k$
$(k\in\bn$) or uniformly approximately factor through $\co$. The first
theorem has a
slightly weaker conclusion still related to factorization properties
but makes no assumption on the random matrices. Indeed, it applies to
operators defined on an arbitrary sequence of Banach spaces. These
results provide information on the closed ideal structure of the
Banach algebra of all operators on the space $\thespace$.
\end{abstract}

\maketitle

\section*{Introduction}

In this paper we study closed ideals of operators on the space
$\thespace$ with the ultimate goal of classifying all of them. When
studying operators on
this space one is quickly
reduced to considering sequences of operators $T^{(m)}\colon
\ell_\infty^m (\ell_1^m)\to \ell_1^m$ ($m\in\bn$), where $\ell_\infty^m
(\ell_1^m)$ is the $\ell_\infty$-sum of $m$ copies of
$\ell_1^m$. Often it will be more convenient to use a different
normalization and
view $T^{(m)}$ as an operator into $L_1=L_1[0,1]$. We shall denote by
$e_{i,j}=e^{(m)}_{i,j}$ the unit vector basis of $\ell_\infty^m
(\ell_1^m)$, where the norm of $\sum_{i,j}a_{i,j}e_{i,j}$ is given by
$\max_i\sum_j\abs{a_{i,j}}$. We then let $T^{(m)}_{i,j}=T^{(m)}_{\phtm{i,j}}(
e^{\phtm{(m)}}_{i,j})$, so $T^{(m)}$ can be identified with the $m\times m$ matrix
$\big(T^{(m)}_{i,j}\big)$ with entries in $L_1$. Our main results
concern such random matrices. The first one is general with no extra
assumptions on the random variables $T^{(m)}_{i,j}$.
\begin{mainthm}
  \label{mainthm:general-dichotomy}
  Let $T^{(m)}\colon \ell_\infty^m (\ell_1^m)\to L_1$ ($m\in\bn$) be a
  uniformly bounded sequence of operators. Then
  \begin{mylist}{(ii)}
  \item
    either the identity operators $\id_{\ell_1^k}\colon
    \ell_1^k\to\ell_1^k$ ($k\in\bn$) uniformly factor through the
    $T^{(m)}$,
  \item
    or the operators $T^{(m)}$ have uniform
    approximate lattice bounds, \ie
    \[
    \V \vare>0\quad \E C>0\quad \V m\in\bn\quad \E g_m\in
  L_1\quad\text{such that}\quad \norm{g_m}_{L_1}\leq C\quad\text{and}
\]
\[
    T^{(m)} \big( B_{\ell_\infty^m (\ell_1^m)}\big) \subset \big\{
    f\in L_1:\,\abs{f}\leq g_m\big\} + \vare B_{L_1}\ .
\]
  \end{mylist}
\end{mainthm}
Here and throughout the paper we denote by $B_X$ the closed unit ball
of a Banach space $X$.
It turns out that this result does not depend on the domain spaces of
the $T^{(m)}$ which can be replaced by an arbitrary sequence of Banach
spaces (\cf Theorem~\ref{thm:general-dichotomy}).  
One of the consequences of this theorem is that the Banach
algebra $\cB(X)$ of all bounded operators on $X=\thespace$ has a
unique maximal ideal. We thus
obtain the following picture of the lattice of closed ideals of
$\cB(X)$. Here $\cK$ is the ideal of compact operators while
$\cG_{\co}$ denotes the ideal of operators factoring through
$\co$. For an operator ideal $\cJ$ we let $\cJb$ be the norm closure
of $\cJ$ and we denote by $\cJ\sur$ the surjective hull of $\cJ$
(defined in Section~\ref{sec:unique-max-ideal}).
\begin{mainthm}
  \label{mainthm:ideal-structure}
  Let $X=\thespace$. We have the following closed ideals in $\cB(X)$:
  \[
  \{0\} \subsetneq \cK(X) \subsetneq \cGb_{\co}(X)
  \subseteq \cGb_{\co}\sur(X) \subsetneq \cB(X)\ .
  \]
  Moreover, if there is another closed ideal $\cJ$ of $\cB(X)$, then
  it must lie between $\cGb_{\co}(X)$ and its surjective
  hull. In particular, $\cGb_{\co}\sur(X)$ is the
  unique maximal ideal of $\cB(X)$.
\end{mainthm}
We do not know whether the inclusion $\cGb_{\co}(X)\subseteq
\cGb_{\co}\sur(X)$ is proper. If it is in fact an equality,
then $\cK(X)$ and $\cGb_{\co}(X)$ are the only non-trivial (\ie
non-zero), proper
closed ideals of $\cB(X)$ and we have a full description of the
lattice of closed ideals of $\cB(X)$. Otherwise
$\cGb_{\co}\sur(X)$ may be the only non-trivial, proper
closed ideal of $\cB(X)$ besides $\cK(X)$ and $\cGb_{\co}(X)$ or
there may also be other new closed ideals strictly between $\cGb_{\co}(X)$
and $\cGb_{\co}\sur(X)$. Classifying the closed ideals of
$\cB(X)$, one is lead to the following problem.
\begin{problem}
  Let $T^{(m)}\colon \ell_\infty^m (\ell_1^m)\to L_1$ ($m\in\bn$) be a
  uniformly bounded sequence of operators. Is it true that
  \begin{mylist}{(ii)}
  \item
    either the identity operators $\id_{\ell_1^k}$ ($k\in\bn$)
    uniformly factor through the $T^{(m)}$,
  \item
    or the $T^{(m)}$ uniformly approximately factor
    through $\ell_\infty^k$ ($k\in\bn$)?
  \end{mylist}
\end{problem}
Our final result gives a positive answer to this problem in the case
when the entries of the matrix associated to $T^{(m)}$ are
independent, symmetric random variables.
\begin{mainthm}
  \label{mainthm:dichotomy-indep-sym-case}
  For each $m\in\bn$ let $T^{(m)}\colon \ell_\infty^m(\ell_1^m)\to
  L_1$ be an operator such that the entries of the corresponding
  random matrix $\big( T^{(m)}_{i,j}\big)$ form a sequence of
  independent, symmetric random variables with
  \[
  \bignorm{T^{(m)}}=\max \bigg\{ \be \Bigabs{\sum_{i=1}^m
  T^{(m)}_{i,j_i}} :\,j_1,\dots,j_m\in\{1,\dots,m\} \bigg\} \leq
  1\ .
  \]
  Then
  \begin{mylist}{(ii)}
  \item
    either the identity operators $\id_{\ell_1^k}$ ($k\in\bn$)
    uniformly factor through the $T^{(m)}$,
  \item
    or the $T^{(m)}$ uniformly approximately factor
    through $\ell_\infty^k$ ($k\in\bn$).
  \end{mylist}
\end{mainthm}
The problem of classifying the closed ideals of operators on a Banach
space goes back to Calkin who in 1941 proved that the compact
operators are the only non-trivial, proper closed ideal in
$\cB(\ell_2)$~\cite{calkin:41}. The same result was later proved for
all $\ell_p$ spaces ($p$ finite) and for $\co$ by Gohberg, Markus, and
Feldman in 1960~\cite{goh-mar-fel:60}. Remarkably, very little is
known about the closed ideals of $\cB(\ell_p\oplus\ell_q)$, and it is
not even known if there are infinitely many of them. For the most recent
results on the spaces $\ell_p\oplus \ell_q$ the reader is invited to
consult~\cite{sari-schlump-tom-troi:07}. 

In the late 1960's Gramsch~\cite{gramsch:67} and Luft~\cite{luft:68}
independently extended Calkin's theorem in a different direction by
classifying all the closed ideals of $\cB(H)$ for each Hilbert
space~$H$ (not necessarily separable). In particular, they showed that
these ideals are well-ordered by inclusion.

It was not until fairly recently that new examples were added to the
list of Banach spaces for which all of the closed ideals of
operators can be determined. In 2004 Laustsen, Loy, and
Read~\cite{laus-loy-read:03} proved that for the Banach space $E =
\big(\bigoplus_{n=1}^\infty\ell_2^n\big)_{\co}$ there are exactly four
closed ideals of $\cB(E)$, namely $\{0\}$, the compact
operators $\cK(E)$, the closure $\cGb_{\co}(E)$ of the set
of operators factoring through $\co$, and $\cB(E)$ itself. A similar
result was subsequently obtained by Laustsen, Schlumprecht and Zs\'ak
for the dual space
$F=\big(\bigoplus_{n=1}^\infty\ell_2^n\big)
_{\ell_1}$~\cite{laus-schlump-zsak:06}. In
2006 Daws~\cite{daws:06} extended Gramsch and Luft's result to the
Gohberg--Markus--Feldman case by classifying the closed ideals of
$\cB(\ell_p(I))$ (for $p$ finite) and $\cB(\co(I))$ where $I$
is an index set of arbitrary cardinality. Again, these ideals are
well-ordered by inclusion. Recently Argyros and Haydon
constructed a space that solves the famous compact-plus-scalar
problem: every operator on their space is a compact perturbation of a
scalar multiple of the identity operator. This remarkable
space has many interesting properties. In particular, as this space
also has a basis, the compact operators are the only non-trivial,
proper closed ideal of the algebra of all operators.

Our paper is organized as follows. In Section~\ref{sec:prelim} we
sketch the proofs of the more straightforward parts of
Theorem~\ref{mainthm:ideal-structure}. We also reduce the ideal
classification problem to the problem stated above (preceding the
statement of Theorem~\ref{mainthm:dichotomy-indep-sym-case}), and we
introduce 
the notions of uniform factorization and uniform approximate
factorization. In Section~\ref{sec:general-dichotomy} we define the
notions of uniform lattice bounds and uniform approximate lattice
bounds, and we prove Theorem~\ref{mainthm:general-dichotomy}. In
Section~\ref{sec:unique-max-ideal} we complete the proof of
Theorem~\ref{mainthm:ideal-structure}. The general dichotomy theorem,
Theorem~\ref{mainthm:general-dichotomy}, gives rise to a very natural
conjecture that would solve the ideal classification problem
completely. In Section~\ref{sec:perturb-unif-latt-bnd} we present a
counterexample to this conjecture. Section~\ref{sec:independent-case}
contains a proof of Theorem~\ref{mainthm:dichotomy-indep-sym-case}.

We use standard Banach space terminology throughout. For convenience
we shall work with real scalars. All our results extend without
difficulty to the complex case. The sign $\abs{\cdot}$ will be used
for absolute value (of a number or a function) as well as for the size
of a finite set. Finally, we denote by $\bi_A$ the indicator function
of a set $A$, and use the probabilistic notation $\bp$ for Lebesgue
measure on $[0,1]$.

\section{Preliminary results}
\label{sec:prelim}

Throughout this paper we fix $X$ to be the Banach space $\thespace$. In
this section we first prove those
parts of Theorem~\ref{mainthm:ideal-structure} that follow easily from
standard basis arguments. We then reduce the problem of finding the
closed ideal structure of $\cB(X)$ to a question about sequences of
operators defined on finite $\ell_\infty$-direct sums of $\ell_1$-spaces with
values in $L_1$ (this reduction will also follow easily from standard
basis arguments). We shall also be introducing definitions and
notations to be used throughout the paper.

We shall only give sketch proofs. The results in this section extend
without difficulty to more general unconditional sums of
finite-dimensional spaces. For detailed proofs in the general case, we
refer the reader to~\cite{laus-loy-read:03}.
\begin{prop}
  \label{prop:ideal-structure-easy}
  We have the following closed ideals in $\cB(X)$:
  \[
  \{0\} \subsetneq \cK(X) \subsetneq \cGb_{\co}(X)
  \subsetneq \cB(X)\ .
  \]
  Moreover, if $T$ is a non-compact operator on $X$, then the closed
  ideal generated by $T$ contains $\cGb_{\co}(X)$. It
  follows that any closed ideal of $\cB(X)$ not in the above list must
  lie strictly between $\cGb_{\co}(X)$ and $\cB(X)$.
\end{prop}
\begin{proof}
  Since $X$ has a basis, the compact operators are the smallest
  non-trivial closed ideal of $\cB(X)$, and the inclusion
  $\cK(X)\subset\cGb_{\co}(X)$ follows. (Note, however, that not every
  compact operator on $X$ factors through $\co$.) This inclusion is
  strict, since $\co$ is complemented in $X$ and a projection onto
  a copy of $\co$ is a non-compact operator in $\cGb_{\co}(X)$. 

  We next show that $\cGb_{\co}(X)\neq \cB(X)$. Recall that if an
  idempotent element of a Banach algebra belongs to the
  closure of an ideal $I$, then in fact it belongs to $I$. Thus, if
  $\cGb_{\co}(X)=\cB(X)$, then the identity operator on $X$ factors through
  $\co$, \ie $X$ is complemented in $\co$, and thus isomorphic to
  it. It is well known, however, that $X$ is not isomorphic to $\co$
  (\eg because $\ell_1$ has cotype~2).

  Finally, let $T$ be a non-compact operator on $X$. To complete the
  proof it is enough to show that the identity on $\co$ factors
  through $T$. Let $(x_n)$ be a bounded sequence in $X$ such that
  $(Tx_n)$ has no convergent subsequence. After passing to a
  subsequence we can assume that both $(x_n)$ and $(Tx_n)$ converge
  coordinatewise (with respect to the obvious basis of $X$).
  We then extract a
  further subsequence for which the difference sequence
  $(Tx_n-Tx_{n+1})$ is bounded away from zero. This way we obtain a
  sequence $(y_n)$ in $X$ such that both $(y_n)$ and $(Ty_n)$ converge
  to zero coordinatewise and $(Ty_n)$ is bounded away from
  zero. We can then pass to a further subsequence such that $(y_n)$
  and $(Ty_n)$ are basic sequences equivalent to the unit vector
  basis of $\co$ and such that their closed linear spans are
  complemented in $X$. It is now straightforward that $\id_{\co}$
  factors through $T$.
\end{proof}
For $n\in\bn$ we let $J_n\colon\ell_1^n\to X$ be the canonical embedding
given by $J_nx=(y_i)$ where $y_n=x$ and $y_i=0$ for $i\neq
n$. For each $m\in\bn$ the map $Q_m\colon X\to\ell_1^m$ denotes
the canonical quotient map defined by $Q_m(y)=y_m$ for $y=(y_i)\in
X$. We introduce projections $P_n=J_nQ_n\in\cB(X)$ for $n\in\bn$, and
$P_A(x)=\sum_{n\in A} P_nx$ for $A\subset\bn$ and $x\in X$.

For an operator $T\colon X\to X$ we let
$T_{m,n}=Q_mTJ_n\colon \ell_1^n\to \ell_1^m$. We can identify $T$ with
the infinite matrix
$\big(T_{m,n}\big)$: if $Tx=y$, then $y_m=\sum_n T_{m,n}x_n$. We say
that $T$ is \emph{locally finite }if the sets
$\{j\in\bn:\,T_{m,j}=0\}$ and $\{i\in\bn:\,T_{i,n}\}$ are finite for
all $m,n\in\bn$, \ie if $T$ has finitely supported rows and columns.
\begin{lem}
  \label{lem:loc-finite-perturbation}
  For any $T\in\cB(X)$ and $\vare>0$ there is a compact operator
  $K\in\cB(X)$ such that $\norm{K}<\vare$ and $T+K$ is locally finite.
\end{lem}
\begin{proof}
  Fix a sequence $(\vare_i)$ in $(0,1)$ with $\sum_i
  \vare_i<\vare$. Let $n\in\bn$. For each $x\in \ell_1^n$ there exists
  $N(n,x)\in\bn$ such that
  $\norm{(I-P_{\{1,\dots,N\}})TJ_nx}<\vare_n/2$ for all $N\geq
  N(n,x)$. By compactness of $B_{\ell_1^n}$, there exists $N_n\in\bn$
  such that
  $\norm{(I-P_{\{1,\dots,N_n\}})TJ_n}<\vare_n$. Then the operator
  $K=\sum_n(I-P_{\{1,\dots,N_n\}})TJ_n$ is compact, $\norm{K}<\vare$ and
  $T-K$ has finite columns.

  Next fix $m\in\bn$. Since the unit vector basis of $\co$ is
  shrinking, for each $f\in\ell_\infty^m$ there exists
  $M(m,f)\in\bn$ such that
  $\norm{fQ_mT(I-P_{\{1,\dots,M\}})}<\vare_m/2$ for all $M\geq
  M(m,f)$. By compactness of $B_{\ell_\infty^m}$, there
  exists $M_m\in\bn$ such that
  $\norm{fQ_mT(I-P_{\{1,\dots,M_m\}})}<\vare_m\norm{f}$ for all
  $f\in\ell_\infty^m$ and hence, by Hahn--Banach,
  $\norm{Q_mT(I-P_{\{1,\dots,M_m\}})}\leq\vare_m$. As before, we now
  obtain a compact operator $K$ such that $\norm{K}<\vare$ and $T-K$
  has finite rows.
\end{proof}
\begin{defn}
Given families $\big( U_i\colon E_i\to F_i\big)_{i\in I}$ and $\big(
V_j\colon G_j\to H_j\big)_{j\in J}$ of operators between Banach
spaces, we say \emph{the $U_i$ uniformly factor through the $V_j$ }(or
that \emph{the $V_j$ uniformly factor the $U_i$}) if
\[
\E C>0\quad \V i\in I\quad \E j_i\in J\,,\ A_i\colon E_i\to
G_{j_i}\,,\ B_i\colon H_{j_i}\to F_i
\]
\[
\text{such that}\quad
U_i=B_iV_{j_i}A_i\quad \text{and}\quad \norm{A_i}\cdot\norm{B_i}\leq
C\ .
\]
We say \emph{the $U_i$ uniformly approximately factor through
  the $V_j$ }(or that \emph{the $V_j$ uniformly approximately factor
  the $U_i$}) if
\[
\V\vare>0\quad \E C>0\quad \V i\in I\quad \E j_i\in J\,,\ A_i\colon
  E_i\to G_{j_i}\,,\ B_i\colon H_{j_i}\to F_i
\]
\[
\text{such
that}\quad \norm{U_i-B_iV_{j_i}A_i}<\vare\quad\text{and}\quad
\norm{A_i}\cdot\norm{B_i}\leq C\ .
\]
\end{defn}
If $G_j=H_j$ and $V_j$ is the identity operator $\id_{G_j}$ on $G_j$
for all $j\in J$, then we will also use the term \emph{factoring
through the }$G_j$ instead of factoring through the $\id_{G_j}$, \etc

For a family $\big( U_i\colon E_i\to F_i\big)_{i\in I}$ of operators
with $\sup_{i\in I} \norm{U_i}<\infty$
we write $\diag(U_i)_{i\in I}$ for the diagonal operator
$\big(\bigoplus _{i\in I} E_i\big)_{\co}\to \big(\bigoplus _{i\in I}
F_i\big)_{\co}$ given by $(x_i)_{i\in I}\mapsto (U_ix_i)_{i\in I}$.

Now let $T\in\cB(X)$ be a locally finite operator. For $m\in\bn$ we
let $R_m$  be the support of the $m^{\text{th}}$ row of $T$: this is
the finite set
$R_m=\{j\in\bn:\,T_{m,j}\neq 0\}$. We set $X_m=\big(\bigoplus_{j\in
  R_m}\ell_1^j\big)_{\ell_\infty}$ and let $J^{(m)}\colon X_m\to X$
and $Q^{(m)}\colon X\to X_m$ be the canonical embedding and quotient
maps given by $J^{(m)}\big( (x_j)_{j\in R_m}\big)=\sum_{j\in R_m}
J_j(x_j)$ and $Q^{(m)}(x)=\big( Q_j(x)\big)_{j\in R_m}$,
respectively. We define $T^{(m)}\colon X_m\to\ell_1^m$ to be the
$m^{\text{th}}$ row of $T$ ignoring the zero entries, \ie $T^{(m)}$
maps $x=(x_j)_{j\in R_m}$ to $Q_mTJ^{(m)}(x)=\sum_{j\in R_m}
T_{m,j}x_j$.

One final piece of notation before we relate factorization properties
of $T$ to those of the sequence $\big(T^{(m)}\big)$: for subsets $A$
and $B$ of $\bn$ we write $A<B$ if $a<b$ for all $a\in A$ and
$b\in B$.
\begin{prop}
  \label{prop:reduction-to-fd-case}
  Let $T\in\cB(X)$ be a locally finite operator. 
  \begin{mylist}{(ii)}
  \item
    If the $T^{(m)}$ uniformly factor the identity operators
    $\id_{\ell_1^k}$ ($k\in\bn$), then $T$ factors the identity
    operator on $X$.
  \item
    $T$ approximately factors through $\co$ if and only if the
    $T^{(m)}$ uniformly approximately factor through $\ell_\infty^n$
    ($n\in\bn$).
  \end{mylist}
\end{prop}
\begin{proof}
  (i)~By the assumption, there exist $C>0$, positive integers $m_1<m_2<\dots$
  and operators $A_k\colon \ell_1^k\to X_{m_k}$ and $B_k\colon \ell_1^{m_k}\to
  \ell_1^k$ such that $\id_{\ell_1^k}=B_kT^{(m_k)}A_k$ and
  $\norm{A_k}\cdot\norm{B_k}\leq C$ for every $k\in\bn$. We may
  assume, after passing to a
  subsequence if necessary, that $R_{m_1}<R_{m_2}<\dots$, so in
  particular the $m_j^{\text{th}}$ and $m_k^{\text{th}}$ rows of $T$
  have disjoint support whenever $j\neq k$. Observe that the identity
  operator
  $\id_X=\diag \big(\id_{\ell_1^k}\big)$ factors through the diagonal
  operator
  \[
  \Tt=\diag (T^{(m_k)})\colon \big( \bigoplus_k X_{m_k}
  \big)_{\co} \rTo \big( \bigoplus_k \ell_1^{m_k}\big)_{\co}\ .
\]
  Indeed,
  we have $\id_X=B\Tt A$, where $A=\mathrm{diag}(A_k)$ and
  $B=\mathrm{diag}(B_k)$. It is therefore sufficient to show that
  $\Tt$ factors through $T$. Define $\At\colon \big(
  \bigoplus_k X_{m_k} \big)_{\co} \to X$ by $(x_k)\mapsto \sum_k
  J^{(m_k)}(x_k)$ and $\Bt\colon X\to \big( \bigoplus_k
  \ell_1^{m_k}\big)_{\co}$ by $x\mapsto
  \big(Q_{m_k}(x)\big)_{k=1}^\infty$. That $\At$ is well-defined
  follows from the assumption $R_{m_1}<R_{m_2}<\dots$. Note that we
  have $\Tt=\Bt T\At$, as required.

  \noindent
  (ii)~Assume the $T^{(m)}$ uniformly approximately factor through
  $\ell_\infty^n$ ($n\in\bn$). Then $\Tt=\mathrm{diag}(T^{(m)})$
  approximately factors through $\big(\bigoplus _k
  \ell_\infty^{n_k}\big)_{\co}$ for some $n_1<n_2<\dots$. This latter space
  is isomorphic to $\co$, so it is enough to observe that $T$ factors
  through $\Tt$. Indeed, $T=\Tt Q$, where $Qx=\big( Q^{(m)}(x)
  \big)$ for $x\in X$.

  The converse implication is clear since each $T^{(m)}$ factors
  through $T$, and $\co$ is a $\cL_\infty$-space.
\end{proof}

\section{The general dichotomy theorem}
\label{sec:general-dichotomy}

In this section we begin our study of factorization properties of
sequences of operators $T_m\colon X_m\to L_1$ ($m\in\bn$) where
$(X_m)$ is a sequence of \emph{arbitrary }Banach spaces. We will prove a
dichotomy theorem in this general setting. In the next section we
shall apply this to an operator $T$ on our space
$X=\thespace$: the $T_m$
will be the rows $T^{(m)}$ of $T$ (as defined before
Proposition~\ref{prop:reduction-to-fd-case}). Before stating our main
theorem we need a definition.
\begin{defn}
  Let $T_i\colon X_i\to L_1$ ($i\in I$) be a family of operators. We
  say the $T_i$ \emph{have uniform lattice bounds }if
\[
\E C>0\quad\V i\in I\quad \E g_i\in
  L_1\quad\text{with}\quad \norm{g_i}_{L_1}\leq C\quad\text{and}\quad
  T_i\big( B_{X_i}\big) \subset \{f\in L_1:\,\abs{f}\leq g_i\}
\]
(\ie $\abs{T_ix}\leq g_i$ for all $x\in B_{X_i}$). The family
$(g_i)_{i\in I}$ is a \emph{uniform lattice bound for the $T_i$}.
 
  We say the $T_i$ \emph{have uniform approximate lattice
    bounds }if
\[
\V\vare>0\ \E C>0\ \V i\in I\ \E g_i\in
    L_1^+\ \text{with}\ \norm{g_i}_{L_1}\leq C\ \text{and}\
T_i\big( B_{X_i}\big) \subset \{f\in L_1:\,\abs{f}\leq g_i\} + \vare
B_{L_1}
\]
(\ie $\bignorm{\big(\abs{T_ix}-g_i\big)^+}_{L_1}\leq\vare$ for all $x\in
  B_{X_i}$). The family $(g_i)_{i\in I}$ is a \emph{uniform
  approximate lattice bound for the $T_i$ corresponding to $\vare$}.
\end{defn}

We now come to one of the main results in this paper, which yields, as
a special case, Theorem~\ref{mainthm:general-dichotomy} stated in the
Introduction.
\begin{thm}
  \label{thm:general-dichotomy}
  Let $T_m\colon X_m\to L_1$ ($m\in\bn$) be a uniformly bounded
  sequence of operators. Then the following dichotomy holds:
  \begin{mylist}{(ii)}
  \item
    either the identity operators $\id_{\ell_1^k}$ ($k\in\bn$)
    uniformly factor through the $T_m$,
  \item
    or the $T_m$ have uniform approximate lattice bounds.
  \end{mylist}
\end{thm}

\begin{rem}
  We observe that this is a genuine dichotomy. Indeed, assume that
  both alternatives hold. By~(i) there exists $C>0$ such that for all
  $k\in\bn$ there is an $m\in\bn$ such that $T_m \big(B_{X_m}\big)$
  contains a sequence
  $f_1,\dots, f_k$ which is $C$-equivalent to the unit vector basis of
  $\ell_1^k$ for some constant $C$ independent of $k$. By a theorem of
  Dor~\cite[Theorem B]{dor:75} there exist $\delta>0$ (depending only
  on $C$) and disjoint sets $E_1,\dots,E_k$ such that
  $\norm{f_j\restrict_{E_j}}\geq \delta$ for all $j$. By~(ii) there
  exists a uniform approximate lattice bound $(g_m)$ for the $T_m$
  corresponding to $\vare=\delta/2$. Then
  \begin{align*}
    \norm{g_m}_{L_1} &\geq \sum _{j=1}^k \norm{g_m\restrict_{E_j}}_{L_1} \geq
    \sum_{j=1}^k \norm{(\abs{f_j}\meet g_m)\restrict_{E_j}}_{L_1}
    \\
    &\geq \sum_{j=1}^k \big( \norm{f_j\restrict_{E_j}}_{L_1} -
    \norm{(\abs{f_j}-g_m)^+\restrict _{E_j}}_{L_1}\big) \geq k\delta/2\ .
  \end{align*}
  Thus $\sup_m \norm{g_m}_{L_1}=\infty$ --- a contradiction.
\end{rem}

Before embarking on the proof of Theorem~\ref{thm:general-dichotomy},
we make a simple observation, which places uniform lattice bounds in the
context of factorization.
\begin{prop}
  \label{prop:lattice-bounds-factorization}
  Let $T_m\colon X_m\to L_1$ ($m\in\bn$) be a uniformly bounded
  sequence of operators.
  \begin{mylist}{(ii)}
  \item
    If the $T_m$ have uniform lattice bounds then they uniformly
    factor through $L_\infty$. In particular, if $\dim X_m<\infty$ for
    all $m$, then the $T_m$ uniformly factor through
    $\ell_\infty^n$ ($n\in\bn$).
  \item
    Suppose that for each $m\in\bn$ we have $X_m=\ell_1^{N_m}$ for some
    $N_m\in\bn$. If the $T_m$ have uniform approximate lattice
    bounds, then they uniformly approximately factor through
    $\ell_\infty^n$ ($n\in\bn$).
  \end{mylist}
\end{prop}
\begin{proof}
  (i)~Let $(g_m)$ be a bounded sequence in $L_1$ such that
  $\abs{T_mx}\leq g_m$ for all $x\in B_{X_m}$ and for all
  $m\in\bn$. Without loss of generality for each $m\in\bn$ we have
  $g_m>0$ everywhere. We can then define maps $A_m\colon X_m\to
  L_\infty$ by $A_mx=\frac{T_mx}{g_m}$ and $B_m\colon L_\infty\to
  L_1$ by $B_mf=g_m\cdot f$. This gives the required factorization
  $T_m=B_mA_m$ with $\sup \norm{A_m}\cdot
  \norm{B_m}=\sup\norm{g_m}_{L_1}<\infty$. The second assertion
  follows immediately by virtue of the fact that $L_\infty$ is a
  $\cL_\infty$-space.

  \noindent
  (ii)~Let $\vare>0$ and let $(g_m)$ be a corresponding uniform
  approximate lattice bound for the $T_m$. For $m\in\bn$ define a
  linear operator $S_m\colon \ell_1^{N_m}\to L_1$ by setting
  $S_me_i=(T_me_i\meet g_m)\join (-g_m)$ ($i=1,\dots,N_m$), where
  $(e_i)_{i=1}^{N_m}$ denotes the unit vector basis of
  $\ell_1^{N_m}$. Then
\[
  \norm{T_m-S_m}=\max _{1\leq i\leq
  N_m} \norm{(T_m-S_m)(e_i)}_{L_1}\leq \vare\ .
\]
Since $(g_m)$ is a uniform lattice
  bound for the $S_m$, it follows from~(i) that the $S_m$
  uniformly factor through $\ell_\infty^n$ ($n\in\bn$).
\end{proof}

We now begin the proof of Theorem~\ref{thm:general-dichotomy}. We will
need two ingredients. The first of these is a sort of converse to the
aformentioned result of Dor~\cite[Theorem B]{dor:75}. This converse
result for an infinite sequence $(f_i)$, from which the
quantitative statement below follows easily, was proved by
H.~Rosenthal~\cite{rosenthal:70} using a combinatorial argument.
Here we sketch a particularly elegant probabilistic
proof from~\cite{john-schec:82} which has the advantage of giving a
linear bound (with respect to~$k$) on the constant $n(\delta,k)$ in
the statement of the theorem.
\begin{thm}
  \label{thm:almost-disjoint-supp-fns}
  For each $\delta>0$ and $k\in\bn$ there exists $n=n(\delta,k)\in\bn$
  such that if $f_1,\dots,f_n$ are functions in $B_{L_1}$ for which
  there are disjoint sets $E_1,\dots,E_n$ with
  $\norm{f_i\restrict_{E_i}}_{L_1}\geq\delta$ for all $i$, then there
  is a subsequence $(f_{j_i})_{i=1}^k$ such that
  \[
  \Bignorm{\sum_{i=1}^k a_if_{j_i}}_{L_1}\geq\frac{\delta}{2}\qquad
  \text{whenever }\sum_{i=1}^k\abs{a_i}=1\ .
  \]
  In particular, $(f_{j_i})_{i=1}^k$ is $\frac2{\delta}$-equivalent to
  the unit vector basis of $\ell_1^k$.
\end{thm}
\begin{proof}
  Fix $\delta\in (0,1]$ and $k\in\bn$. Let
  $n=\bigintp{\frac{10}{\delta}}\cdot k$, and let $A=(\alpha_{i,j})$ be the
  $n\times n$ matrix with
  $\alpha_{i,j}=\norm{f_i\restrict_{E_j}}_{L_1}$ when $i\neq j$ and zeros
  on the diagonal. Note that the row sums of $A$ satisfy $\sum_{j=1}^n
  \alpha_{i,j}\leq \norm{f_i}_{L_1}\leq 1$. We will show the existence of a
  $k\times k$ submatrix $(\alpha_{i,j})_{i,j\in F}$ whose row sums are at
  most $\frac{\delta}2$. An easy direct computation then shows that the
  subsequence $(f_i)_{i\in F}$ has the required property.

  Pick a subset $E$ of $\{1,\dots,n\}$ of size $2k$ uniformly at
  random. Then
\[
\be \sum_{i,j\in E} \alpha_{i,j} = \be \sum_{i,j=1}^n \alpha_{i,j} \bi_{\{i,j\in
  E\}} = \sum_{i,j=1}^n
  \alpha_{i,j} \binom{n-2}{2k-2}\binom{n}{2k}^{-1} \leq \frac{(2k)^2}{n-1}\
  .
\]
It follows that for some subset $E$ the row sums of the submatrix
$(\alpha_{i,j})_{i,j\in E}$ are at most $\frac{2k}{n-1}$ on average. Hence,
by Markov's inequality, at least half of the rows sum to at most twice
this average. \Ie for some $F\subset E$ with $\abs{F}=k$, the row sums
of $(\alpha_{i,j})_{i,j\in F}$ are at most~$\frac{\delta}{2}$.
\end{proof}

The second ingredient is a theorem of Dor which shows, in particular,
that a subspace of $L_1$ whose
Banach--Mazur distance to $\ell_1^k$ is not too large is
well complemented.

\begin{thm}[({Dor~\cite[Theorem~A]{dor:75}})]
  \label{thm:dor}
  Let $\mu$ and $\nu$ be measures and $T\colon L_1(\nu)\to L_1(\mu)$
  an isomorphic embedding with
  $\norm{T}\cdot\norm{T^{-1}}=\lambda<\sqrt{2}$. Then there is a
  projection $P$ of $L_1(\mu)$ onto the range of $T$ with
  \[
  \norm{P}\leq \big( 2\lambda^{-2}-1\big)^{-1}\ .
  \]
\end{thm}
In the proof of Theorem~\ref{thm:general-dichotomy} we shall use an
argument that will also be needed in
Section~\ref{sec:independent-case}, so we state and prove it
separately.
\begin{prop}
  \label{prop:almost-disjoint-supp-implies-factorization}
  Let $T_m\colon X_m\to L_1$ ($m\in\bn$) be operators with
  $\norm{T_m}\leq 1$ for all $m\in\bn$. Assume that there exists
  $\delta>0$ such that for all $n\in\bn$ there exist $m\in\bn$,
  functions $f_1,\dots,f_n\in T_m(B_{X_m})$ and pairwise disjoint sets
  $E_1,\dots,E_n$ such that $\norm{f_i\restrict_{E_i}}_{L_1}\geq\delta$ for
  all $i$. Then the identity operators $\id_{\ell_1^k}$ uniformly
  factor through the $T_m$.
\end{prop}
\begin{proof}
  By Theorem~\ref{thm:almost-disjoint-supp-fns} we can deduce the
  following from the assumption:
  \begin{multline}
    \label{eq:fixing-ell_1^k's}
    \V k\in\bn\quad \E m\in\bn\quad \E y_1,\dots,y_k\in
  B_{X_m}\quad \text{such that}\\[1ex]
    \Bignorm{\sum_{i=1}^k
  a_iT_my_i}_{L_1}\geq\frac{\delta}2\qquad\text{whenever
  }\sum_{i=1}^k\abs{a_i}=1\ .
  \end{multline}
  Thus, in particular, $T_m\big(B_{X_m}\big)$ contains a sequence
  $\frac2{\delta}$-equivalent to the unit vector basis of
  $\ell_1^k$. We next use a well known argument of
  James (see \eg \cite[Proposition~2]{os:03})
 to improve the equivalence constant $\frac2{\delta}$. Fix
  $1<\lambda<\sqrt{2}$. Choose $r\in\bn$ such that
  $\big(\frac2{\delta}\big)^{1/r}<\lambda$, and then set
  $K=k^{r}$. By~\eqref{eq:fixing-ell_1^k's} there exist $m\in\bn$
  and $y_1,\dots,y_K\in B_{X_m}$ such that 
  \begin{equation}
    \label{eq:fixing-ell_1^K's}
    \Bignorm{\sum_{i=1}^K a_iT_my_i}_{L_1}\geq\frac{\delta}2\qquad
    \text{whenever }\sum_{i=1}^K\abs{a_i}=1\ .
  \end{equation}
  Now James's argument shows that there is a block basis
  $z_j=\sum_{i=p_{j-1}+1}^{p_j} a_iy_i$, where $0=p_0<p_1<\dots <p_k=K$
  and $\sum_{i=p_{j-1}+1}^{p_j}\abs{a_i}=1$ for all~$j$, such that
  $\big( T_mz_j\big)_{j=1}^k$ is
  $\big(\frac2{\delta}\big)^{1/r}$-equivalent to the unit vector
  basis of $\ell_1^k$. Thus there exist constants $0<\alpha\leq\beta$
  with $\frac{\beta}{\alpha}<\lambda$ such that
  \begin{equation}
    \label{eq:fixing-good-ell_1^k's}
    \alpha\leq \Bignorm{\sum_{j=1}^k b_jT_mz_j}_{L_1}\leq\beta\qquad
    \text{whenever }\sum_{j=1}^k\abs{b_j}=1\ .
  \end{equation}
  Note that by~\eqref{eq:fixing-ell_1^K's} we have
  $\beta\geq\norm{T_mz_j}_{L_1}\geq\frac{\delta}2$. 
  Now define $A_m\colon\ell_1^k \to X_m$ by $e_j\mapsto z_j$. We
  then have $\bignorm{T_mA_m}\cdot\bignorm{(T_mA_m)^{-1}}<\lambda$, so
  we can apply Theorem~\ref{thm:dor}: there is a projection $P$ of
  $L_1$ onto the
  range of $T_mA_m$ with $\norm{P}\leq \big(
  2\lambda^{-2}-1\big)^{-1}$. Let $B_m\colon L_1\to\ell_1^k$ be
  the composition of $P$ with the map $\spn\{
  T_mz_j:\,j=1,\dots,k\}\to\ell_1^k$ defined by $T_mz_j\mapsto
  e_j$. Using~\eqref{eq:fixing-good-ell_1^k's} and the above
  estimates involving $\alpha$ and $\beta$, we obtain
  \[
  \norm{A_m}\leq 1\ ,\quad\norm{B_m}\leq \norm{P}\cdot
  \frac1{\alpha}\leq \norm{P}\cdot\lambda\cdot \frac2{\delta}\leq
  \frac{2\lambda}{\delta}\cdot \big( 2\lambda^{-2}-1\big)^{-1}\ ,
  \]
  and $\id_{\ell_1^k}=B_mT_mA_m$. Thus the $T_m$ uniformly factor the
  identity operators $\id_{\ell_1^k}$ ($k\in\bn$), as required.
\end{proof}

\begin{proof}[Proof of Theorem~\ref{thm:general-dichotomy}]
  Without loss of generality we have $\norm{T_m}\leq 1$ for all
  $m$. We assume that~(ii) fails: there exists an $\vare>0$ such that
  for all $C>0$ there exists $m\in\bn$ such that
  \begin{equation}
    \label{eq:no-unif-approx-lattice-bound}
    \V g\in L_1^+\text{ with }\norm{g}_{L_1}\leq C\ \E x\in B_{X_m}
    \text{ such that }\bignorm{\big(\abs{T_mx}-g\big)^+}_{L_1}>
    \vare\ .
  \end{equation}
  From this we deduce that the assumption of
  Proposition~\ref{prop:almost-disjoint-supp-implies-factorization} is
  satisfied with $\delta=\vare/2$.

  Fix $n\in\bn$ and set $N=\bigintp{\frac{4n^2}{\vare}}$. Putting
  $C=N-1$, we find
  $m\in\bn$ such that~\eqref{eq:no-unif-approx-lattice-bound}
  holds. From now on we let $T=T_m$. Successive applications
  of~\eqref{eq:no-unif-approx-lattice-bound} yield $x_1,\dots,x_N\in
  B_{X_m}$ such that
  \[
  \Bignorm{\Big( \abs{Tx_i}-\bigjoin_{1\leq j<i}
    \abs{Tx_j}\Big)^+}_{L_1} > \vare\qquad\text{for }i=1,\dots,N\ .
  \]
  (Note that $\bignorm{\bigjoin_{1\leq j<i}\abs{Tx_j}}_{L_1}\leq
  N-1=C$ for all~$i\leq N$.) For each $i=1,\dots,N$ set
  \begin{eqnarray*}
    D_i & = &\Big\{ \omega\in[0,1]:\,\abs{Tx_i}(\omega)>\bigjoin_{1\leq
      j<i}\abs{Tx_j}(\omega)\Big\}\ ,\text{ and}\\
    \Dt_i &=& \Big\{
    (\omega,t)\in[0,1]\times\br:\,\omega\in D_i,\ \abs{Tx_i}(\omega)>t>\bigjoin_{1\leq
      j<i}\abs{Tx_j}(\omega) \Big\}\ .
  \end{eqnarray*}
  (Thus $\Dt_i$ is the region between the graphs of $\abs{Tx_i}$ and
  $\bigjoin_{1\leq j<i}\abs{Tx_j}$ where the former is greater.) For
  each $1<i_0\leq N$, the regions
  $(D_{i_0}\times\br)\cap\Dt_i\,,\ i=1,\dots,i_0-1$, are pairwise disjoint
  and lie beneath the graph of $\abs{Tx_{i_0}}$. It follows that
  \[
  \sum_{i=1}^{i_0-1} \Bignorm{\Big( \abs{Tx_i}-\bigjoin_{1\leq j<i}
    \abs{Tx_j}\Big)^+\cdot \bi_{D_{i_0}}}_{L_1} \leq
  \norm{Tx_{i_0}}_{L_1}\leq 1\ ,
  \]
  and hence
  \[
  \Bigabs{\Big\{ i<i_0:\, \Bignorm{\Big( \abs{Tx_i}-\bigjoin_{1\leq
        j<i} \abs{Tx_j}\Big)^+\cdot
      \bi_{D_{i_0}}}_{L_1}\geq\frac{\vare}{2n} \Big\}} \leq
  \frac{2n}{\vare}\ .
  \]
  By the choice of $N$, we can therefore find $N=i_1>i_2>\dots
  >i_n\geq 1$ such that
  \[
  \Bignorm{\Big( \abs{Tx_{i_s}}-\bigjoin_{1\leq
      j<i_s}\abs{Tx_j}\Big)^+\cdot \bi_{D_{i_r}}}_{L_1}
  <\frac{\vare}{2n} \qquad\text{for }1\leq r<s\leq n\ .
  \]
  Now set $f_s=Tx_{i_s}$ and $E_s=D_{i_s}\setminus \bigcup_{r<s}
  D_{i_r}$ for $s=1,\dots,n$. Then $f_1,\dots,f_n\in T(B_{X_m})$, the
  sets $E_1,\dots,E_n$ are pairwise disjoint, and $\norm{f_i\restrict
  _{E_i}}_{L_1}\geq\frac{\vare}2$ for all $i=1,\dots,n$. This
    completes the proof of the theorem.
\end{proof}

\section{The existence of a unique maximal ideal}
\label{sec:unique-max-ideal}

Let $\cJ$ be an operator ideal. We say $\cJ$ is \emph{injective
}if, given any operator $T\colon E\to F$ between Banach spaces and an
(isomorphic) embedding $J\colon F\to G$, we have $JT\in\cJ(E,G)$ implies
$T\in\cJ(E,F)$. The \emph{injective hull }of $\cJ$ is defined to be
\[
  \cJ\inj(E,F) = \big\{ T\in\cB(E,F):\,\E \text{ embedding
  }J\colon F\to G\text{ such that }JT\in\cJ(E,G) \big\}\ .
\]
It is easy to see that $\cJ\inj$ is an injective
operator ideal and it is the smallest injective ideal containing
$\cJ$.

The dual concept is that of a surjective ideal. We say $\cJ$ is \emph{
  surjective }if, given any operator $T\colon E\to F$ and a quotient
map (\ie an onto bounded linear map) $Q\colon D\to E$, we have $TQ\in\cJ(D,F)$ implies
$T\in\cJ(E,F)$. The \emph{ surjective hull }of $\cJ$ is
\[
  \cJ\sur(E,F) = \big\{ T\in\cB(E,F):\,\E \text{ quotient
    map }Q\colon D\to E\text{ such that }TQ\in\cJ(D,F) \big\}\ .
\]
One can again verify that $\cJ\sur$ is a surjective
operator ideal and it is the smallest such ideal containing $\cJ$.

In this section we investigate what happens if we apply these two ways
of obtaining a new ideal from a given one in the algebra
$\cB(X)$. Recall that throughout
$X=\thespace$. Since $\cK$ is an
injective and surjective operator ideal, we only need to consider
$\cGb_{\co}(X)$. Taking the injective hull, we obtain nothing new.
\begin{thm}
  $\cGb_{\co}\inj(X)=\cB(X)$.
\end{thm}
\begin{proof}
  Since $X$ is the $\co$-sum of finite-dimensional spaces, we have an
  embedding $J\colon X\to\co$ and $JI_X\in\cG_{\co}(X,\co)$.
\end{proof}
 The surjective hull, however, does give new information about the
 ideal structure of $\cB(X)$. This is the main result of this section.
\begin{thm}
  \label{thm:unique-max-ideal}
  $\cGb_{\co}\sur(X)$ is the unique maximal ideal of
  $\cB(X)$.
\end{thm}

\begin{proof}
  We first show that $\cGb_{\co}\sur(X)$ is a proper
  ideal. Assume, for a contradiction, that this ideal contains
  $\id_X$, \ie that some quotient map $Z\to X$ approximately factors
  through $\co$. Without loss of generality we can assume that $Z$ is
  separable. By considering a quotient map $\ell_1\to Z$, we may also
  assume that $Z=\ell_1$, so there is an
  embedding $X^*=\big(\bigoplus_{n=1}^\infty\ell_\infty^n
  \big)_{\ell_1}\to \ell_\infty$ which approximately factors through
  $\ell_1$. It follows easily that $\ell_1$ contains
  $\ell_\infty^n$ ($n\in\bn$) uniformly. This is impossible, \eg
  because $\ell_1$ has cotype~2.

  Now fix $T\in\cB(X)$. We are going to show that if $\id_X$ does not
  factor through $T$, then $T$ belongs to
  $\cGb_{\co}\sur(X)$. This will prove that every proper ideal
  is contained in $\cGb_{\co}\sur(X)$, and our proof is then
  complete.

  Without loss of generality we can assume that $T$ is locally finite
  (Lemma~\ref{lem:loc-finite-perturbation}). We are going to use the
  notation introduced before
  Proposition~\ref{prop:reduction-to-fd-case}:
  $R_m=\{j\in\bn:\,T_{m,j}\neq 0\}$ is the $m^{\text{th}}$ row support
  of $T$, $X_m=\big(\bigoplus_{j\in R_m}\ell_1^j\big)_{\ell_\infty}$,
  and $T^{(m)}\colon X_m\to\ell_1^m$ is the $m^{\text{th}}$ row of
  $T$.

  Fix quotient maps
  $\pi\colon\ell_1\to X$ and $\pi_m\colon \ell_1^{N_m}\to X_m$ with
  \[
  \frac12 B_{X_m} \subset \pi_m \big(B_{\ell_1^{N_m}}\big) \subset
  B_{X_m}\qquad (m\in\bn).
  \]
  Note that $\pit=\diag(\pi_m)\colon \Big( \bigoplus_m \ell_1^{N_m}
  \Big)_{\co}\to \big( \bigoplus_m X_m \big)_{\co}$ is also a quotient
  map. 

  Recall from the proof of Proposition~\ref{prop:reduction-to-fd-case}(ii)
  that $T$ factors through $\Tt=\diag (T^{(m)})$ via the map $Q\colon
  X\to\big( \bigoplus_m X_m \big)_{\co}$ given by $Qx=\big( Q^{(m)}(x)
  \big)_{m=1}^\infty$ for $x\in X$. By the lifting property of $\ell_1$
  there is a map $\Qt\colon\ell_1\to \Big( \bigoplus_m \ell_1^{N_m}
  \Big)_{\co}$ with $\norm{\Qt}\leq 2$ such that $Q\pi=\pit\Qt$. We
  thus have the following commuting diagram:
  \[
  \begin{diagram}
    \ell_1 & \rTo^{\pi} & X & &\\
    \dTo_{\Qt} & & \dTo_{Q} & \rdTo^{T} &\\
    \Big( \bigoplus_m \ell_1^{N_m} \Big)_{\co} & \rTo^{\pit} & \Big(
    \bigoplus_m X_m \Big)_{\co} & \rTo^{\Tt} & \Big( \bigoplus_m
    \ell_1^m \Big)_{\co}\ .
  \end{diagram}
  \]
  We claim that $T\pi$ approximately factors through
  $\co$. Since $T$ does not factor $\id_X$, the $T^{(m)}$ do not factor
  $\id_{\ell_1^k}\ (k\in\bn)$ uniformly
  (Proposition~\ref{prop:reduction-to-fd-case}(i)). By
  Theorem~\ref{thm:general-dichotomy},
  the $T^{(m)}$, and hence the $T^{(m)}\pi_m$, have uniform
  approximate lattice bounds. It follows by
  Proposition~\ref{prop:lattice-bounds-factorization}(ii) that the
  $T^{(m)}\pi_m$ uniformly approximately factor through
  $\ell_\infty^n\ (n\in\bn)$. This implies that $\Tt\pit$
  approximately factors through $\co$, and hence so does $T\pi$.
\end{proof}

\begin{rem}
  Of course, we have $\cGb_{\co}(X)\subset
  \cGb_{\co}\sur(X)$, but we do not know whether this
  inclusion is strict \ie whether there exist closed ideals of
  $\cB(X)$ other than those listed in
  Proposition~\ref{prop:ideal-structure-easy}.
\end{rem}

\section{Perturbing operators with uniform approximate lattice bounds}
\label{sec:perturb-unif-latt-bnd}

In Proposition~\ref{prop:lattice-bounds-factorization}(ii), can we
replace $\ell_1^{N_m}$ by more general spaces $X_m$? \Ie given
operators $T_m\colon X_m\to L_1\ (m\in\bn)$ with uniform
approximate lattice bounds, do the $T_m$ uniformly approximately
factor through $\ell_\infty^n\ (n\in\bn)$?
Proposition~\ref{prop:lattice-bounds-factorization}(i) gives an
affirmative answer to this question \emph{provided }there exist
arbitrarily small perturbations of the $T_m$ with uniform lattice
bounds. This leads to the following question.
\begin{question}
  Let $T_m\colon X_m\to L_1\ (m\in\bn)$ be a uniformly bounded
  sequence of operators. Assume that the $T_m$ have uniform
  approximate lattice bounds. Does there exist, for all $\vare>0$,
  a sequence $S_m\colon X_m\to L_1\ (m\in\bn)$ of operators
  with $\norm{T_m-S_m}<\vare$ for all $m$ such that the $S_m$ have uniform
  lattice bounds?
\end{question}
One cannot hope for a positive answer for a general sequence $(X_m)$
of Banach spaces: \eg the diagonal operators
$A_m\colon\ell_2^m\to\ell_1^m$ below give a simple counterexample (\cf
Proposition~\ref{prop:perturb-simple-example}). However, the proof of
Proposition~\ref{prop:lattice-bounds-factorization}(ii) shows that we
do have a positive answer in the case when each $X_m$
is an $\ell_1$-space. We would
hope to generalize this to the case when each $X_m$ is a finite
$\ell_\infty$-direct sum of finite-dimensional $\ell_1$-spaces. A
positive answer in that case together with
Theorem~\ref{thm:general-dichotomy} and
Proposition~\ref{prop:lattice-bounds-factorization}(ii) would provide
a positive answer to the problem raised in the Introduction (stated
before Theorem~\ref{mainthm:dichotomy-indep-sym-case}). In turn, this would imply
(by Proposition~\ref{prop:reduction-to-fd-case}) that the list in
Proposition~\ref{prop:ideal-structure-easy} is a complete list of the
closed ideals of $\cB(X)$ for our space $X=\thespace$.

In this section we present an example that shows that the above question
has a negative answer even in the case when each $X_m$ is a
finite-dimensional $\ell_\infty$-space. Here it will be convenient to
use a different normalization: the range spaces will be $\ell_1^m\
(m\in\bn)$ instead of $L_1$.

For $m\in\bn$ set $N_m=2^m$ and $X_m=\ell_\infty^{N_m}$. Let
$r^m_i\in X_m^*$ be the $i^{\text{th}}$ Rademacher function, $i=1,\dots,m$,
normalized with respect to the $\ell_\infty$-norm, \ie the coordinates
of each $r^m_i$ are $\pm 1$. Let
$e^m_i,\ i=1,\dots,m$, denote the standard basis of $\br^m$. Define
$T_m\colon X_m\to\ell_1^m$ by defining its adjoint
\[
T_m^*\colon\ell_\infty^m \rTo X_m^*\ ,\qquad e^m_i\rMapsto
\frac1{\sqrt{m}}\frac1{N_m} r^m_i\ ,\qquad i=1,\dots,m\ .
\]
Thus we have
\[
\ip{T_m x}{e^m_i} = \ts\frac1{\sqrt{m}N_m}
\ds\bigip{x}{r^m_i}\ ,\qquad x\in\ell_\infty^{N_m},\ i=
1,\dots,m\ .
\]
Note that $\norm{T_m}\leq 1$ for all~$m$. We now show that the $T_m$
have uniform approximate lattice bounds. We have factorizations
\[
\begin{diagram}
  X_m &            & \rTo^{T_m} &            & \ell_1^m  \\
      & \rdTo_{B_m} &           & \ruTo_{A_m} &          \\
      &            &  \ell_2^m &             &          \\
\end{diagram}
\]
obtained from its dual
\[
\begin{diagram}
  \ell_\infty^m &             & \rTo^{T_m^*} &              & X_m^*\\
               & \rdTo_{A_m^*} &            & \ruTo_{B_m^*} &      \\
               &              &  \ell_2^m  &              &      \\
\end{diagram}
\]
where $A_m^* (e^m_i)=\frac1{\sqrt{m}}e^m_i$ and
$B_m^*(e^m_i)=\frac1{N_m}r^m_i$ for $i= 1,\dots,m$. Note that
$\norm{A_m}=1$ (consider extreme points of
$B_{\ell_\infty^m}$) and $\norm{B_m}= 1$ for all $m\in\bn$. Thus,
in particular, it is sufficient to show that the $A_m$ have
uniform approximate lattice bounds.

\begin{prop}
\label{prop:simple-example}
  Given $\vare>0$, let $C=\frac1{\vare}$. Then for each $m\in\bn$ and
  $x=\sum_{i=1}^m x_i e^m_i\in B_{\ell_2^m}$ we have
  $\bignorm{A_mx\restrict_L}_{\ell_1^m} \leq \vare$, where
  $L=L(m,x)=\{ i:\,\abs{\ip{A_mx}{e^m_i}}>C/m\}$.
\end{prop}
\begin{proof}
  For $x\in B_{\ell_2^m}$ we have $L=L(m,x)=\big\{
  i:\,\abs{x_i}>C/\sqrt{m}\big\}$. Since
  $\abs{L}\frac{C^2}{m}\leq\norm{x}_{\ell_2^m}^2$, by Cauchy--Schwarz we
  get
  \[
    \bignorm{A_mx\restrict_L}_{\ell_1^m} = \sum _{i\in L}
    \frac{\abs{x_i}}{\sqrt{m}} \leq \sqrt{\frac{\abs{L}}{m}}\cdot
\norm{x}_{\ell_2^m} \leq \frac1{C}=\vare\ .\qedhere
\]
\end{proof}
This shows that for any $\vare>0$ and for $C=\frac1{\vare}$ we have
\[
T_m(B_{X_m})\subset \big\{ y={\ts \sum
y_ie^m_i}\in\ell_1^m:\, \abs{y_i}\leq \tfrac{C}{m}\text{ for
}i=1,\dots,m \big\} +\vare B_{\ell_1^m}\ .
\]
Thus the $T_m$ have uniform approximate lattice bounds. The
difficult part is to show that the $T_m$ cannot be perturbed to get
uniform lattice bounds. We first show this for the $A_m$. Although we
do not need this, the proof is much simpler than for the $T_m$ and
contains some of the ideas used later.

\begin{prop} 
\label{prop:perturb-simple-example}
  Let $\vare\in(0,1)$. Assume that for all $m\in\bn$ there exist
  $S_m\colon \ell_2^m\to\ell_1^m$ and $g_m\in\ell_1^m$ such that
  \begin{equation}
    \label{eq:S-approx-A} \norm{S_m-A_m} < \vare
    \end{equation}
\begin{equation}
    \label{eq:g-bounds-S} \abs{S_mx} \leq g_m\quad\text{for all }x\in
    B_{\ell_2^m}\ .
    \end{equation}
  Then $\sup \norm{g_m}_{\ell_1^m}=\infty$.
\end{prop}

\begin{proof}
  Fix $m\in \bn$. We will show that $\norm{g_m}_{\ell_1^m}\geq
  \frac{(1-\vare)\sqrt{m}}3$. For the rest of the proof we drop
  the subscript~$m$; $\pi$ will denote a permutation of
  $\{1,\dots,m\}$ as well as the corresponding linear map on $\br^m$
  given by $e_i\mapsto e_{\pi(i)}$. Note that $A=\pi^{-1}A\pi$ for
  all $\pi$. Let
  \[
  \Sb=\frac1{m!} \sum_{\pi} \pi^{-1}S\pi\qquad\text{and}\qquad
  C=\norm{g}_{\ell_1^m}=\sum _{i=1}^m g(i)\ .
  \]
  Then $\norm{\Sb-A}<\vare$ and
  \[
  \bigabs{\ip{\Sb x}{e_i}}\leq \frac1{m!} \sum_{\pi}
  \bigabs{\ip{S\pi (x)}{\pi (e_i)}} \leq \frac1{m!} \sum_\pi g(\pi(i))
  =\frac{C}{m}\ .
  \]
  Thus, without of loss of generality, $g$ is the constant function
  $\frac{C}{m}$ and $S=\pi^{-1}S\pi$ for all $\pi$. It follows
  that for some $a,b\in\br$ we have $\ip{Se_i}{e_i}=\frac{a}{m}$
  for all~$i$, and $\ip{Se_i}{e_j}=\frac{b}{m(m-1)}$ for all
  $i\neq j$.

  Now by~\eqref{eq:g-bounds-S} we have $\abs{a}\leq C$ and
  $\abs{b}\leq C(m- 1)$. We next apply~\eqref{eq:S-approx-A} to
  $x= \frac1{\sqrt{m}}\sum (-1)^i e_i$ to obtain
  \[
    \vare  > \norm{Ax-Sx}_{\ell_1^m} \geq \norm{Ax}_{\ell_1^m} - \norm{Sx}_{\ell_1^m}
    \geq \ts 1 - \frac1{\sqrt{m}}\Big( \frac{\abs{a}}{m} +
    \frac{2\abs{b}}{m(m-1)} \Big)\cdot m \geq 1
  - \frac{3C}{\sqrt{m}}\ ,
  \]
  from which our claim follows.
\end{proof}

\begin{rem}
  The motivation behind the proof of
  Proposition~\ref{prop:perturb-simple-example} is as follows. In
  contrast to~$A$, $S$ cannot be large on the diagonal because it has
  a lattice bound. On the other hand, being close to $A$, $S$ has norm
  close to~1, so the off-diagonal entries of~$S$ must make a
  significant contribution to the norm of~$S$. Next, since $A$ is
  symmetric, we could ``symmetrize'' $S$, and hence assume that~$S$ is
  constant off the diagonal. Applying $S$ to a
  ``flat'' vector whose coefficients alternate in sign, we produce a small
  vector due to cancellations. On the other hand, when we apply the
  diagonal operator $A$ to the same vector, no cancellations occur
  making the outcome large. This contradicts that $A$ and $S$ are
  close in norm. The idea behind the proof of
  Theorem~\ref{thm:no-perturb-with-unif-bound} below is exactly the
  same.
\end{rem}

We now turn to the proof that the $T_m$ cannot be perturbed to have
uniform lattice bounds. By Khintchine's inequality in $L_1$ (see, for
example~\cite{jameson:87}), with $K=\sqrt{2}$ we have
\begin{equation}
  \label{eq:khintchin}
  \frac1{K} \Big(\sum _{i=1}^m a_i^2\Big)^{1/2} \leq
  \Bignorm{\sum_{i=1}^m a_i \ts\frac1{N_m}r^m_i}_{\ell_1^{N_m}} \ds\leq \Big(\sum
  _{i=1}^m a_i^2\Big)^{1/2}\qquad\text{for all
  }(a_i)_{i=1}^m\in\br^m\ .
\end{equation}
\begin{thm}
  \label{thm:no-perturb-with-unif-bound}
  Let $0<\vare<\frac1{4K}$. Assume that for all $m\in\bn$ there exist
  $g_m\in\ell_1^m$ and $S_m\colon\ell_\infty^{N_m}\to\ell_1^m$ such
  that
  \begin{equation}
    \norm{S_m-T_m}<\vare\ ,\label{eq:counter-perturb}
    \end{equation}
  \begin{equation}
    \abs{S_mx}\leq g_m\quad\text{for all }x\in
    B_{\ell_\infty^{N_m}}\ .\label{eq:counter-exact-bound}
    \end{equation}
  Then $\sup_m\norm{g_m}_{\ell_1^m}=\infty$.
\end{thm}
\begin{proof}
  We shall argue by contradiction. Assume that for some
  $0<\vare<\frac1{4K}$ there is a $C>0$ such that for all $m\in\bn$
  there exist $g_m\in\ell_1^m$ and
  $S_m\colon\ell_\infty^{N_m}\to\ell_1^m$ such
  that~\eqref{eq:counter-perturb} and~\eqref{eq:counter-exact-bound}
  hold, and moreover $\norm{g_m}_{\ell_1^m}\leq C$ for all $m\in\bn$.

  We will obtain a contradiction in a number of steps. From now on we
  fix a large $m$ (to be specified at the end of the proof), and drop $m$
  in the various subscripts and superscripts. We denote by
  $N$ the power set of $\{1,\dots,m\}$ and write the
  standard basis of $\br^N$ as $e_\alpha,\ \alpha\in
  N$. The Rademacher functions can then be expressed as
  \[
  r_i=\sum _{\alpha,\,i\in\alpha} e_\alpha -\sum _{\alpha,\,i\notin\alpha}e_\alpha\qquad
  i=1,\dots,m\ .
  \]
  The letter $\pi$ will always denote a permutation of $\{1,\dots,m\}$
  as well as the following induced maps:
  \[
  \begin{diagram}
    \ell_1^m & \rTo^{\pi} & \ell_1^m\ , & e_i & \rMapsto &
    e_{\pi(i)}\\
    N & \rTo^{\pi} & N\ , & \alpha & \rMapsto & \{
    \pi(i):\,i\in\alpha\}\\
    \ell_\infty^N & \rTo^{\pi} & \ell_\infty^N\ ,& e_\alpha & \rMapsto
    & e_{\pi(\alpha)}\ .
  \end{diagram}
  \]
  Note that the first and third interpretations of $\pi$ are
  isometries. The letter $R$ will also stand for a number of different
  maps:
  \[
  \begin{diagram}
    \ell_1^m & \rTo^{R} & \ell_1^m\ , & e_i & \rMapsto & -e_{i}\\
    N & \rTo^{R} & N\ ,& \alpha & \rMapsto & \neg\alpha=
    \{1,\dots,m\}\setminus\alpha\\
    \ell_\infty^N & \rTo^{R} & \ell_\infty^N\ , & e_\alpha & \rMapsto
    & e_{R(\alpha)}\ .
  \end{diagram}
  \]
  Here again $R$ is an isometry in the first and third definitions.
  Note also that the last map satisfies $R\big( r_i\big)=
  -r_i$, and that $R$ and $\pi$ commute in each their interpretations.
 
  Having fixed our notation, we next show that $S$ can be assumed to
  have various symmetries. We begin with the observation that $T$ is
  symmetric in the sense that it equals the composite $\pi^{-1}T\pi$:
  \[
  \begin{diagram}
    \ell_\infty^N & \rTo^{\pi} & \ell_\infty^N & \rTo^{T} & \ell_1^m &
    \rTo^{\pi^{-1}} & \ell_1^m\ .
  \end{diagram}
  \]
  Similarly, $T= RTR$. Set $\Sb=
  \frac1{m!}\sum_{\pi} \pi^{-1} S\pi$ and $C=
  \norm{g}_{\ell_1^m}=\sum_{i=1}^m g(i)$. Then $\norm{\Sb-
  T}<\vare$, and, by~\eqref{eq:counter-exact-bound}, for all $x\in
  B_{\ell_\infty^N}$ and for $i= 1,\dots,m$ we have
  \[
  \bigabs{\ip{\Sb x}{e_i}}\leq \frac1{m!} \sum_{\pi}
  \bigabs{\ip{S\pi (x)}{\pi (e_i)}} \leq \frac1{m!} \sum_\pi
  g(\pi(i))=\frac{C}{m}\ .
  \]
  Thus, without of loss of generality, we may assume that $g$ is the
  constant function $\frac{C}{m}$ and that $S= \pi^{-1}S\pi$
  for all $\pi$.

  Next we set $\Sb= \frac12 (S+ RSR)$. Then
  $\norm{\Sb- T}<\vare$, $\bigabs{\ip{\Sb x}{e_i}}\leq
  \frac{C}{m}$ for all $x\in B_{\ell_\infty^N}$ and for $i=
  1,\dots,m$. We can thus also assume that $S= RS R$.

  The above two symmetrization procedures have the following
  implications for the matrix of $S$: there
  exist $a_k\in\br,\ k= 1,\dots,m$, such that
  \[
  S_{i,\alpha} = \ip{S e_\alpha}{e_i} =%
  \begin{cases}
  a_{\abs{\alpha}} & \text{if }i\in\alpha\\
  -a_{\abs{\neg\alpha}} & \text{if }i\notin\alpha
  \end{cases}
  \qquad\qquad \alpha\in N,\ i= 1,\dots,m\ .
  \]
  To complete the proof of
  Theorem~\ref{thm:no-perturb-with-unif-bound} we require a number of
  lemmas.
  \begin{lem}
    \label{lem:counter:bound-on-matrix-coeffs}
  $\ds 2\sum_{k=1}^m \abs{a_k} \binom{m-1}{k-1} \leq \frac{C}{m}$.
  \end{lem}
\begin{proof}
  For $x\in \ell_\infty^N$ and $i=1,\dots,m$ we have
  \begin{equation}
    \label{eq:counter:action-of-S}
    \ip{Sx}{e_i} = \sum_\alpha x_\alpha S_{i,\alpha} = \sum _{k=1}^m
    a_k \sum_{\abs{\alpha}=k,\ i\in\alpha} (x_\alpha -
    x_{\neg\alpha})\ .
  \end{equation}
  Fix an arbitrary $i\in \{1,\dots,m\}$, set
  \[
  x_\alpha =%
  \begin{cases}
    \sgn (a_k) & \text{if }\abs{\alpha}= k,\ i\in\alpha\\
    -\sgn (a_k) & \text{if }\abs{\alpha}= m-k,\ i\notin\alpha\ ,
  \end{cases}
  \]
  and use~\eqref{eq:counter-exact-bound} to obtain
  \[
  \bigabs{\ip{Sx}{e_i}}= 2\sum_{k=1}^m \abs{a_k} \binom{m-1}{k-1} \leq
  \frac{C}{m}\ ,
  \]
  as required.
\end{proof}

\begin{lem}
  \label{lem:counter:norm-of-tx}
  Fix $k_0\in\bn$. Let $\vare_i\in\{-1,+1\},\ i= 1,\dots,m$. For
  $\alpha\in N$ set
  \[
  x_\alpha = \sgn\Big( \sum_{i\in\alpha} \vare_i -%
  \sum_{i\notin\alpha} \vare_i\Big)
  \]
  whenever $k_0\leq\abs{\alpha}\leq m- k_0$ (we let $\sgn(0)=0$),
  otherwise set $x_\alpha= 0$. Then there exists $m(k_0)\in\bn$
  such that $\norm{T x}_{\ell_1^m}\geq \frac1{4K}$ provided $m\geq m(k_0)$.
\end{lem}
\begin{proof}
  Recall that $T^*\colon\ell_\infty^m\to \ell_1^N$ is given by
  $T^*(e_i)= \frac1{\sqrt{m}N} r_i$, $i= 1,\dots,m$. For
  $y= \sum_{i=1}^m \vare_ie_i$ Khintchine's
  inequality~\eqref{eq:khintchin} yields
  \[
  \norm{T^* y}_{\ell_1^N} = \Bignorm{\sum _{i=1}^m \ts \frac{\vare_i}{\sqrt{m}N}
  r_i}_{\ell_1^N} \geq \ds\frac1{K}\Bignorm{\sum_{i=1}^m \ts
    \frac{\vare_i}{\sqrt{m}} e_i}_{\ell_2^m} = \ds\frac1{K}\ .
  \]
  It follows that setting $z=\sgn\big(T^*y\big)$, we have
  \[
  \norm{Tz}_{\ell_1^m} \geq \ip{Tz}{y} = \ip{z}{T^*y}
  = \norm{T^*y}_{\ell_1^N} \geq \frac1{K}\ .
  \]
  Now for any $\alpha\in N$ we have
  \[
    \ip{T^*y}{e_\alpha} = \ts\frac1{\sqrt{m}N}\ds\sum
  _{i=1}^m \vare_i \ip{r_i}{e_\alpha}
  = \ts\frac1{\sqrt{m}N}\ds \Big( \sum_{i\in\alpha} \vare_i
  -\sum_{i\notin\alpha} \vare_i\Big)\ ,
  \]
and hence
\[
z_\alpha = \sgn\Big( \sum_{i\in\alpha} \vare_i -\sum_{i\notin\alpha}%
\vare_i\Big)\ .
\]
Note that $x_\alpha=z_\alpha$ whenever $k_0\leq\abs{\alpha}\leq
m- k_0$.

Observe that if we add an element to the set
$\alpha\in N$, then the expression $\sum_{i\in\alpha} \vare_i
-\sum_{i\notin\alpha} \vare_i$ changes by at most~$2$ in absolute
value. It follows that
\[
\sum _{\abs{\alpha}=k+1} \bigabs{\ip{T^*y}{e_\alpha}} \geq \sum
_{\abs{\alpha}=k} \bigabs{\ip{T^*y}{e_\alpha}} - \binom{m}{k}
\frac{2}{\sqrt{m}N}
\]
whenever $0\leq k< \frac{m}2$. (Indeed, there exists an injection
from sets of size~$k$ to sets of size~$k+ 1$ mapping each
$\alpha$ to some  set $\beta\supset\alpha$. This can be seen using
Hall's marriage theorem.) Iterating $k_0$ times, we get
\begin{align*}
  \sum _{\abs{\alpha}=k+k_0} \bigabs{\ip{T^*y}{e_\alpha}} &\geq \sum
  _{\abs{\alpha}=k} \bigabs{\ip{T^*y}{e_\alpha}} - \sum
  _{j=0}^{k_0-1} \binom{m}{k+j} \frac{2}{\sqrt{m}N}\\
  &\geq \sum _{\abs{\alpha}=k} \bigabs{\ip{T^*y}{e_\alpha}} -
  \frac2{\sqrt{m}}
\end{align*}
whenever $0\leq k< \frac{m}2- k_0$. Summing over~$k$, we obtain
\[
\sum_{k=k_0}^{2k_0-1} \sum _{\abs{\alpha}=k}
\bigabs{\ip{T^*y}{e_\alpha}} \geq \sum _{k=0}^{k_0-1}
\sum_{\abs{\alpha}=k} \bigabs{\ip{T^*y}{e_\alpha}} -
\frac{2k_0}{\sqrt{m}}
\]
provided $k_0<\frac{m}4$. Similarly (or using
$\ip{T^*y}{e_{\neg\alpha}}= -\ip{T^*y}{e_\alpha}$), we obtain
\[
\sum_{k=k_0}^{2k_0-1} \sum _{\abs{\alpha}=m-k}
\bigabs{\ip{T^*y}{e_\alpha}} \geq \sum _{k=0}^{k_0-1}
\sum_{\abs{\alpha}=m-k} \bigabs{\ip{T^*y}{e_\alpha}} -
\frac{2k_0}{\sqrt{m}}\ .
\]
Putting these together, we finally get
\begin{align*}
  \norm{T x}_{\ell_1^m} & \geq \ip{Tx}{y} = \sum_{k_0\leq\abs{\alpha}\leq%
    m-k_0} \bigabs{\ip{e_\alpha}{T^*y}} \\[1ex]
  & \geq \frac13 \sum_\alpha \bigabs{\ip{e_\alpha}{T^*y}} -%
  \frac{4k_0}{3\sqrt{m}} > \frac1{4K}
\end{align*}
provided $m$ is sufficiently large.
\end{proof}
The quantity $d(m,k)$ in Lemmas~\ref{lem:counter:coords-of-sx}
and~\ref{lem:counter:size-of-coords-of-sx}  is defined for an even
integer $m$ as follows:
  \[
  d(m,k)=%
  \begin{cases}
    \ds\binom{\frac{m}{2}-1}{\frac{k-1}{2}}^2 &
    \text{if $k$ is odd}\\[3ex]
    \ds\binom{\frac{m}{2}-1}{\frac{k}2-1}
    \binom{\frac{m}{2}-1}{\frac{k}2}
    & \text{if $k$ is even.}
  \end{cases}
  \]
\begin{lem}
  \label{lem:counter:coords-of-sx}
  Fix $k_0\in\bn$, let $m\in\bn$ be even, and set $\vare_i= (-1)^i$
  for $i= 1,\dots,m$. Define $x= (x_\alpha)\in\ell_\infty^N$ as
  in Lemma~\ref{lem:counter:norm-of-tx}. Then for $k_0\leq
  k\leq m-k_0$ and for $j= 1,\dots,m$ we have
  \[
  \sum_{\abs{\alpha}=k,\ j\in\alpha} x_\alpha= (-1)^j\cdot d(m,k)\ .
  \]
\end{lem}
\begin{proof}
  It is sufficient to consider $j= m$. Let $E$ be the set of all
  even numbers in $\{1,\dots,m\}$. Note that for the given
  choice of signs $\vare_1,\dots,\vare _m$ we have
  \[
  x_\alpha = \sgn\Big( \sum_{i\in\alpha} \vare_i -%
  \sum_{i\notin\alpha} \vare_i\Big) = \sgn\Big( \sum_{i\in\alpha}
  \vare_i \Big)\ .
  \]
  Given $\alpha\in N$ with $\abs{\alpha}=k$ and $m\in\alpha$, let
  \[
  \beta= \{ i+ 1:\,i\leq m- 2,\ i\in\alpha\setminus E\}
  \cup \{i- 1:\,i\leq m- 2,\ i\in\alpha\cap E\} \cup
  \big( \alpha\cap\{ m- 1,m\} \big)\ .
  \]
  Then $\abs{\beta}= k$, $m\in\beta$ and $x_\alpha+
  x_\beta=0$ unless $m- 1\notin\alpha$ and either ($k$ is
  odd and) $\abs{\alpha\cap E}=\frac{k+1}2$, or ($k$ is even and)
  $\abs{\alpha\cap E}=\frac{k}2$ or $\frac{k}2+ 1$. The result
  follows.
\end{proof}

\begin{lem}
\label{lem:counter:size-of-coords-of-sx}
Let $m\in\bn$ be even. Then
\[
d(m,k) \leq
2\binom{m-1}{k-1}\binom{k}{\bigintp{\frac{k}2}}\frac1{2^k}\qquad\text{for
each }k= 1,\dots,m\ .
\]
\end{lem}
\begin{proof}
  Assume $k$ is even. Then
  \begin{multline*}
    d(m,k) \binom{m-1}{k-1}^{-1} =\\[2ex]
    \frac1{2^{k-1}} \cdot
    \frac{\big[(m-2)(m-4)\dots(m-k+2)\big]\cdot\big[(m-2)(m-4)\dots(m-k)\big]
    }{(m-1)(m-2)\dots (m-k+1)}
\cdot \frac{(k-1)!}{\big(\frac{k}2-1\big)!\cdot \big(
      \frac{k}2\big)!} \\[2ex]
    = \frac{m-2}{m-1}\cdot\frac{m-4}{m-3}\cdots \frac{m-k}{m-k+1}\cdot
    \binom{k}{k/2}\cdot \frac1{2^k}\leq\binom{k}{k/2}\cdot \frac1{2^k} \ .
  \end{multline*}
An almost identical computation works for odd $k$ except we get an
extra factor of~$2$ in that case.
\end{proof}

\begin{lem}
\label{lem:counter:binom-coeffs}
There is a universal constant $U$ such that
\[
\binom{k}{\bigintp{\frac{k}2}} \leq U\frac{2^k}{\sqrt{k}}\qquad\text{for all
}k\in\bn\  .
\]
\end{lem}
\begin{proof}
  For $\frac{k}2- \sqrt{k}\leq j<\frac{k}2$ we have
  \[
  \binom{k}{j+1} = \binom{k}{j}\cdot\frac{k-j}{j+1} \leq
  \binom{k}{j}\cdot\frac{\frac{k}2+\sqrt{k}}{\frac{k}2-\sqrt{k}} \leq
  \binom{k}{j}\cdot \Big(1 + \frac6{\sqrt{k}}\Big)
  \]
  provided $k$ is sufficiently large. It follows that for
  $\frac{k}2- \sqrt{k}\leq j<\frac{k}2$ we have
  \[
  \binom{k}{\bigintp{\frac{k}2}} \leq \Big(1 +
  \frac6{\sqrt{k}}\Big)^{\sqrt{k}} \cdot \binom{k}{j} \leq \me^6 \cdot
  \binom{k}{j}
  \]
  for sufficiently large $k$. Hence for a universal constant~$U$ and
  for all $k\in\bn$ we have
  \[
  \sqrt{k}\cdot \binom{k}{\bigintp{\frac{k}2}} \leq U\cdot 2^k\ ,
  \]
  as required.
\end{proof}

\noindent
\emph{Proof of Theorem~\ref{thm:no-perturb-with-unif-bound} continued.}
We finally have all the ingredients to obtain the required
contradiction. Choose $k_0, m\in\bn$ with
$\frac{2UC}{\sqrt{k_0}}<\frac1{4K}-\vare$, $m\geq m(k_0)$ and $m$ even. Recall
that $C$ and $\vare$ were fixed at the very beginning of the proof,
$K$ is the Khintchine constant, $U$ is the universal constant obtained
in Lemma~\ref{lem:counter:binom-coeffs} above, and $m(k_0)$ is given
by Lemma~\ref{lem:counter:norm-of-tx}.

Let $x=
(x_\alpha)\in\ell_\infty^N$ be as in
Lemma~\ref{lem:counter:norm-of-tx} with $\vare_i=
(-1)^i$. Note that $x_{\neg\alpha}=-x_\alpha$ for all $\alpha\in
N$. We have
\begin{align*}
  \norm{Sx}_{\ell_1^m} &= \sum _{i=1}^m \bigabs{\ip{Sx}{e_i}}\\
  &\leq \sum_{i=1}^m \sum _{k=1}^m 2\abs{a_k}\cdot
  \Bigabs{\sum_{\abs{\alpha}=k,\ i\in\alpha} x_\alpha}
  &&\text{by~\eqref{eq:counter:action-of-S}}\\
  & \leq m\sum _{k_0\leq k\leq m-k_0} 2\abs{a_k} d(m,k)
  &&\text{by Lemma~\ref{lem:counter:coords-of-sx}}\\
  & \leq m\sum _{k_0\leq k\leq m-k_0} 4\abs{a_k} \binom{m-1}{k-1} U
  \frac1{\sqrt{k}}&& \text{by
    Lemmas~\ref{lem:counter:size-of-coords-of-sx}
    and~\ref{lem:counter:binom-coeffs}}\\
  & \leq \frac{2UC}{\sqrt{k_0}} < \frac1{4K}-\vare
  &&\text{by Lemma~\ref{lem:counter:bound-on-matrix-coeffs}.}
\end{align*}
Finally, by Lemma~\ref{lem:counter:norm-of-tx} we have
$\norm{Tx}_{\ell_1^m}\geq\frac1{4K}$, and this
contradicts~\eqref{eq:counter-perturb}.
\end{proof}

\section{Searching for new ideals}
\label{sec:independent-case}

Proposition~\ref{prop:reduction-to-fd-case} tells us that a possible
new closed ideal in $\cB(X)$ (if there is one) is generated by an operator
defined by a sequence $T^{(m)}\colon \ell_\infty^m (\ell_1^m)\to L_1$
($m\in\bn$) which neither factors the identity operators
$\id_{\ell_1^k}$ ($k\in\bn$) uniformly, nor does it factor through
$\ell_\infty^k$ ($k\in\bn$) approximately uniformly. The main result
of this
section, stated as Theorem~\ref{mainthm:dichotomy-indep-sym-case} in
the Introduction, shows that there is no such sequence when for each
$m\in\bn$, the entries of the random matrix $\big(T^{(m)}_{i,j}\big)$
are independent, symmetric random variables.

We begin with a characterization of sequences of operators which uniformly
factor through $\ell_\infty^k$ ($k\in\bn$) in terms of the
$2$-summing norm. The 2-summing norm is defined for an operator
$U\colon E\to F$ between Banach spaces as
\begin{multline*}
  \pi_2(U)=\sup\Big\{ \big({\ts\sum_{s=1}^k
    \norm{Uz^{(s)}}^2}\big)^{1/2}:\,
  k\in\bn,\ z^{(1)},\dots,z^{(k)}\in E\ ,\\
  {\ts\sum_{s=1}^k \abs{\ip{z^{(s)}}{z^*}}^2\leq 1}\quad \V z^*\in
  B_{E^*} \Big\}\ .
\end{multline*}
We denote by $\Omega_k$ the probability space
$\big(\{1,\dots,k\},\mu_k\big)$, where $\mu_k$ is the uniform
probability measure given by $\mu_k\big(\{i\}\big)=\frac1k$ for
$i=1,\dots,k$.

\begin{thm}
  \label{thm:factorization-2-summing}
  Let $T^{(m)}\colon \ell_\infty^m (\ell_1^m)\to L_1$ ($m\in\bn$) be a
  uniformly bounded sequence of operators. Then the following are
  equivalent.
  \begin{mylist}{(iii)}
  \item
    The $T^{(m)}$ uniformly factor through $\ell_\infty^k$
    ($k\in\bn$).
  \item
    $\sup_m \pi_2\big(T^{(m)}\big) < \infty$.
  \item
    The $T^{(m)}$ uniformly factor through the formal identity maps
    \[
    \iota_k\colon \ell_\infty^k\to L_2(\Omega_k)\ ,\quad\ts \sum_{i=1}^k
    x_ie_i\mapsto \sum_{i=1}^k x_i\bi_{\{i\}}\qquad (k\in\bn).
    \]
  \end{mylist}
\end{thm}
\begin{proof}
  (i)$\Rightarrow$(ii) follows from the fact that $\pi_2(\cdot)$ is an
  ideal norm and from the following consequence of Grothendieck's
  theorem (\cf~\cite[Theorem~3.5]{dies-jar-tong:95}).
  \begin{thm}
    \label{thm:grothendieck}
    Let $\Phi$ be a compact, Hausdorff space and $\mu$ an arbitrary
    measure on some measurable space. Then for $1\leq p\leq 2$ any
    operator $U\colon C(\Phi)\to L_p(\mu)$ is $2$-summing with
    $\pi_2(U)\leq K_G\norm{U}$, where $K_G$ is the Grothendieck
    constant.
  \end{thm}
  \noindent
  (ii)$\Rightarrow$(iii) is a consequence of the following special
  case of Pietsch's Factorization Theorem
  (\cf~\cite[Corollary~2.16]{dies-jar-tong:95}).
  \begin{thm}
    Let $E$ and $F$ be Banach spaces, and let $\Phi$ be a w$^*$-compact
    subset of $B_{E^*}$ which is $1$-norming for $E$. Let
    $\kappa\colon E\to C(\Phi)$ denote the canonical embedding:
    $\kappa(x)(x^*)=x^*(x),\ x\in E,\ x^*\in\Phi$.

    Then an operator $u\colon E\to F$ is $2$-summing if and only if
    there is a probability measure $\mu$ on the Borel $\sigma$-algebra
    of $\Phi$ and an operator $\ut\colon L_2(\mu)\to F$ such that
    $u=\ut\circ\iota\circ\kappa$, where $\iota\colon C(\Phi)\to
    L_2(\mu)$ is the formal identity. Moreover, $\ut$ can be chosen
    with $\norm{\ut}=\pi_2(u)$.
  \end{thm}
  \noindent
  (iii)$\Rightarrow$(i) is of course obvious.  
\end{proof}

Recall that for each $m\in\bn$ we denote by $e_{i,j}=e^{(m)}_{i,j}$
the unit vector basis of $\ell_\infty^m (\ell_1^m)$ such that the norm
of $\sum_{i,j}a_{i,j}e_{i,j}$ is given by
$\max_i\sum_j\abs{a_{i,j}}$. We identify an operator
$U\colon \ell_\infty^m(\ell_1^m)\to L_1$ with the $m\times m$ matrix
$(U_{i,j})$ in $L_1$, where $U_{i,j}=U(e_{i,j})$. We now estimate
$\pi_2(U)$ in the case the matrix entries $U_{i,j}$ form a symmetric
sequence of random variables. Here and elsewhere we will make use of
the \emph{square function inequality:} if $f_1,\dots,f_n\in L_1$ form
a symmetric sequence of random variables, then
\begin{equation}
  \label{eq:square-function}
  \frac1K \Bignorm{\Big(\sum_{i=1}^n \abs{f_i}^2\Big)^{1/2}}_{L_1}
  \leq\Bignorm{\sum_{i=1}^nf_i}_{L_1} \leq \Bignorm{\Big(\sum_{i=1}^n
    \abs{f_i}^2\Big)^{1/2}}_{L_1}\ .
\end{equation}
This is a well known consequence of Khintchine's
inequality~\eqref{eq:khintchin}.
\begin{lem}
  \label{lem:pi-2-norm}
  Let $m\in\bn$, and let $U\colon \ell_\infty^m(\ell_1^m)\to L_1$ be
  an operator such that the matrix entries $U_{i,j}$ form a symmetric
  sequence of random variables. Then
  \[
  \pi_2(U)\leq \Big(\sum_{i=1}^m \max_{1\leq j\leq m}
  \norm{U_{i,j}}_{L_2}^2\Big)^{1/2}\ .
  \]
\end{lem}
\begin{proof}
  By definition
  \begin{equation}
    \label{eq:def-of-pi-2-of-U}
    \pi_2^2(U)= \sup_{(z^{(s)})_{s=1}^k}
    \frac{\sum_{s=1}^k\norm{Uz^{(s)}}_{L_1}^2}{\sup_{z^*\in
        B_{\ell_1^m(\ell_\infty^m)}} \sum_{s=1}^k
      \abs{\ip{z^{(s)}}{z^*}}^2}
  \end{equation}
  where the supremum is over all $k\in\bn$ and
  $z^{(1)},\dots,z^{(k)}\in \ell_\infty^m(\ell_1^m)$. We will estimate
  the denominator and numerator of the above expression separately. We
  will denote by $\rho$ an arbitrary element $(\rho_j)_{j=1}^m$ of
  $\{\pm 1\}^m$. We begin with the denominator:
  \[
    \sup_{z^*\in B_{\ell_1^m(\ell_\infty^m)}}\ \sum_{s=1}^k
    \abs{\ip{z^{(s)}}{z^*}}^2 = \max_{1\leq i\leq m} \max_{\rho}
    \sum_{s=1}^k \Bigabs{\sum_{j=1}^m \rho_j z^{(s)}_{i,j}}^2
    \geq \max_{1\leq i\leq m} \sum_{s=1}^k \sum_{j=1}^m
    \abs{z^{(s)}_{i,j}}^2\ .
  \]
  The equality follows since the sup is attained at an extreme point of
  $B_{\ell_1^m(\ell_\infty^m)}$. We then replace $\max_\rho$ by
  $\Ave_\rho$, interchange $\Ave_\rho$ and $\sum_{s=1}^k$, and compute the
  variance of a linear combination of independent Bernoulli random
  variables. This yields the inequality. Now the numerator:
  \begin{align*}
    \sum_{s=1}^k\norm{Uz^{(s)}}_{L_1}^2 & \leq\sum_{s=1}^k \Bignorm{\Big(
      \sum_{i,j=1}^m \bigabs{z^{(s)}_{i,j}U_{i,j}}^2
      \Big)^{1/2}}_{L_1}^2  \leq \sum_{s=1}^k \Bignorm{\sum_{i,j=1}^m
      \bigabs{z^{(s)}_{i,j}U_{i,j}}^2}_{L_1} \\[2ex]
    &= \sum_{s=1}^k \sum_{i,j=1}^m \abs{z^{(s)}_{i,j}}^2 \norm{U_{i,j}}_{L_2}^2
  = \sum_{i=1}^m \bigg( \sum_{s=1}^k \sum_{j=1}^m \abs{z^{(s)}_{i,j}}^2
    \norm{U_{i,j}}_{L_2}^2 \bigg)\\[2ex]
    &\leq \sum_{i=1}^m \max_{1\leq j\leq m}\norm{U_{i,j}}_{L_2}^2\cdot
    \bigg( \sum_{s=1}^k \sum_{j=1}^m \abs{z^{(s)}_{i,j}}^2 \bigg)\ .
  \end{align*}
  Here the first inequality is the square function
  inequality~\eqref{eq:square-function}, the second inequality follows
  from Jensen's inequality whereas the rest is straightforward.
  Substitution of our estimates into~\eqref{eq:def-of-pi-2-of-U}
  yields the result.
\end{proof}

\begin{proof}[Proof of Theorem~\ref{mainthm:dichotomy-indep-sym-case}]
  For $m\in\bn$ we let $\cF_m$ be the set of functions
  $\{1,\dots,m\}\to\{1,\dots,m\}$. Functions $j,j'\in\cF_m$ are said
  to be \emph{disjoint }if $j_i\neq j'_i$ for all $i=1,\dots,m$. Since
  $\norm{T^{(m)}}$ is attained at an extreme point of
  $B_{\ell_\infty^m(\ell_1^m)}$, we have
  \[
  \norm{T^{(m)}}=\sup\Big\{ \be \bigabs{\sum_{i=1}^m \rho_i
  T^{(m)}_{i,j_i}}:\, j\in\cF_m,\ \rho\in\{\pm1\}^m \Big\}\ .
  \]
  By the symmetry of the $T^{(m)}_{i,j}$, we in fact have
  \[
  \norm{T^{(m)}}=\sup\Big\{ \be \bigabs{\sum_{i=1}^m
  T^{(m)}_{i,j_i}}:\, j\in\cF_m\Big\}\ .
  \]
  We consider two cases motivated by the notion of uniform
  approximate lattice bounds. The second case is the negation of the
  first.
  \begin{mylist}{(ii')}
  \item[(i')]
    $\E\vare>0\ \V C>0\ \V n\in\bn\ \E m\in\bn$ and
    pairwise disjoint functions $j^{(s)}\in \cF_m$ ($s=1,\dots,n$)
    such that
    \begin{equation}
      \label{eq:caseA}
      \Bignorm{\sum_{i=1}^m T^{(m)}_{i,j^{(s)}_i}\cdot\bi
        _{\big\{\bigabs{T^{(m)}_{i,j^{(s)}_i}}>C\big\}}}_{L_1}
      \geq\vare\qquad \text{for }s=1,\dots,n\ .
    \end{equation}
  \item[(ii')]
    $\V \vare>0\quad \E C>0\quad \E n\in\bn\quad \V m\geq n$ there
  exist pairwise
    disjoint functions $j^{(s)}\in \cF_m$ ($s=1,\dots,n$) such that
    \begin{equation}
      \label{eq:caseB}
      \Bignorm{\sum_{i=1}^m
        T^{(m)}_{i,j_i}\cdot\bi_{\{\abs{T^{(m)}_{i,j_i}}>C\}}}_{L_1} <\vare
    \end{equation}
    for each $j\in \cF_m$ that is disjoint from all the $j^{(s)}$.\\
  \end{mylist}
  We will deduce alternatives~(i) and~(ii) of
  Theorem~\ref{mainthm:dichotomy-indep-sym-case} from the above
  cases~(i') and~(ii'), respectively.
  We begin with case~(i'). Fix $n\in\bn$ and choose $C>0$
  such that $\big(1-\frac2C\big)^{n}\geq\frac12$. Now case~(i') gives
  $m\in\bn$ and pairwise disjoint functions $j^{(s)}\in \cF_m$
  ($s=1,\dots,n$) such that~\eqref{eq:caseA} holds. To avoid cumbersome
  notation, we assume, after permuting entries in each row if
  necessary, that $j^{(s)}_i=s$ for all $i=1,\dots,m$ and
  $s=1,\dots,n$. We also drop the superscript $m$ from $T^{(m)}$ for
  the rest of this case.

  Fix $s\in\{1,\dots,n\}$. We apply the square function
  inequality~\eqref{eq:square-function} twice and monotonicity of
  $\norm{\cdot}_{L_1}$ to~\eqref{eq:caseA}, to obtain
  \begin{align*}
    \Bignorm{\sum_{i=1}^m T_{i,s}\cdot\bi_{\{\max_{i'}
        \abs{T_{i',s}}>C\}}}_{L_1} & \geq \frac1{K} \Bignorm{\Big(
      \sum_{i=1}^m T^2_{i,s}
      \cdot\bi_{\{\max_{i'}\abs{T_{i',s}}>C\}}\Big)^{1/2}}_{L_1}\\
    & \geq \frac1K \Bignorm{\Big(
      \sum_{i=1}^m T^2_{i,s}
      \cdot\bi_{\{\abs{T_{i,s}}>C\}}\Big)^{1/2}}_{L_1}\\
    & \geq \frac1K \Bignorm{\sum_{i=1}^m T_{i,s}\cdot\bi_{\{
        \abs{T_{i,s}}>C\}}}_{L_1} \geq \frac{\vare}{K}\ .
  \end{align*}
  Now set $f_s=\sum_{i=1}^m T_{i,s}$, $E'_s=\big\{
  \max_i\abs{T_{i,s}}>C\big\}$ and $E_s=E'_s\cap\bigcap_{r\neq
    s}(E'_r)^{\complement}$.
  We have $\norm{f_s}_{L_1}=\be \bigabs{\sum_{i=1}^m T_{i,s}} \leq 1$
  and $\norm{f_s\restrict_{E'_s}}_{L_1}\geq\frac{\vare}{K}$. By an
  inequality of L\'evy (\cf~\cite[Proposition~2.3]{led-tal:91}) and Markov's
  inequality we have
  \[
  \prob{E'_s}=\bigprob{\max_i\abs{T_{i,s}}>C}\leq
  2\cdot \bigprob{\bigabs{{\ts\sum_{i=1}^m T_{i,s}}}>C} \leq
  \frac2C\ .
  \]
  Since the $T_{i,j}$ are independent, it follows that
  \[
    \norm{f_s\restrict_{E_s}}_{L_1} = \be \bigabs{f_s \bi_{E_s'}\cdot
      \bi_{\bigcap_{r\neq s}(E_r')^{\complement}}}
    =\be \bigabs{f_s \bi_{E_s'}} \cdot \Bigprob{\bigcap_{r\neq
        s}(E_r')^{\complement}} \geq \frac{\vare}K \cdot
    \Big(1-\frac2C\Big)^{n-1}\geq \frac{\vare}{2K}\ .
  \]
  Thus we have proved that for all $n\in\bn$ there exist $m\in\bn$,
  $f_1,\dots,f_n\in T^{(m)}\big( B_{\ell_\infty^m(\ell_1^m)}\big)$ and
  disjoint sets $E_1,\dots,E_n$ with
  $\norm{f_s\restrict_{E_s}}\geq\frac{\vare}{2K}$ for
  $s=1,\dots,n$. By
  Proposition~\ref{prop:almost-disjoint-supp-implies-factorization}
  the identity maps $\id_{\ell_1^k}$ ($k\in\bn$) uniformly factor through the
  $T^{(m)}$.
 
  We now turn to case~(ii'). Fix $\vare>0$ and choose the
  corresponding $C>0$ and $n\in\bn$. We will show that for every
  $m\in\bn$ there exists $S^{(m)}\colon \ell_\infty^m(\ell_1^m)\to
  L_1$ such that $\norm{T^{(m)}-S^{(m)}}<\vare$, and moreover
  $\sup_m\pi_2(S^{(m)})<\infty$. We can then complete the proof by
  applying Theorem~\ref{thm:factorization-2-summing} to deduce that
  the $S^{(m)}$ uniformly factor through $\ell_\infty^k$
  ($k\in\bn$). Since $\vare$ was arbitrary, it follows that the
  $T^{(m)}$ uniformly approximately factor through $\ell_\infty^k$
  ($k\in\bn$).
  
  Fix $m\in\bn$. If $m<n$, then we can take $S^{(m)}=T^{(m)}$. So
  assume $m\geq n$, put $T=T^{(m)}$, $\cF=\cF_m$, and let
  $j^{(s)}\in \cF$ ($s=1,\dots,n$) be pairwise disjoint functions
  such that~\eqref{eq:caseB} holds for each $j\in \cF$ that is
  disjoint from all the $j^{(s)}$. We may again assume for
  convenience of notation that $j^{(s)}$ is the constant function
  with value $s$ for each $s=1,\dots,n$. We now define
  \[
  S=S^{(1)}+S^{(2)} \colon \ell_\infty^m(\ell_1^m)\to L_1
  \]
  by letting, for each $i=1,\dots,m$,
  \begin{align*}
    S^{(1)}_{i,j} &= \begin{cases}
      T_{i,j} & \text{if }1\leq j\leq n\\
      0 & \text{if }n<j\leq m
    \end{cases}\\[1ex]
    S^{(2)}_{i,j} &=\begin{cases}
    0 & \text{if }1\leq j\leq n\\
    T_{i,j}\cdot \bi_{\big\{ \abs{T_{i,j}}\leq C\big\}} & \text{if
    }n<j\leq m\ .
    \end{cases}
  \end{align*}
  We first check that $\norm{T-S}<\vare$. Here the suprema are taken
  over all $j\in\cF$ and $\rho\in\{\pm1\}^m$.
  \begin{align*}
    \norm{T-S} &=   \sup_{j, \rho}\be
    \Bigabs{\sum_{i=1}^m \rho_i (T-S)_{i,j_i}}\\
    &=    \sup_{j, \rho}\be \Bigabs{\sum_{i:\,n<j_i} 
      \rho_iT_{i,j_i}\cdot \bi_{\big\{
        \abs{T_{i,j_i}}>C\big\}}}<\vare\ .
  \end{align*}
  The first line comes from looking at the extreme points of
  $B_{\ell_\infty^m(\ell_1^m)}$. The second line follows from the
  definition of $S$ and~\eqref{eq:caseB} as well as the use of convexity
  and the symmetry of the $T_{i,j}$.

  We next estimate $\pi_2(S)$ from above. First, $S^{(1)}$ clearly
  factors through $\ell_\infty^m(\ell_1^n)$ with constant~$1$. Since
  $\ell_\infty^m(\ell_1^n)$ is $n$-isomorphic to $\ell_\infty^{mn}$,
  it follows by Theorem~\ref{thm:grothendieck} that
  $\pi_2(S^{(1)})\leq K_G\cdot n$. Second, we can estimate
  $\pi_2(S^{(2)})$ as follows. First, by Lemma~\ref{lem:pi-2-norm} we
  have
  \[
  \pi_2^2(S^{(2)}) \leq \sum_{i=1}^m \max_{1\leq j\leq m}
  \bignorm{S^{(2)}_{i,j}}_{L_2}^2 = \max_{j\in\cF} \sum_{i=1}^m
  \bignorm{S^{(2)}_{i,j_i}}_{L_2}^2= \max_{j\in\cF}
  \Bignorm{\sum_{i=1}^m S^{(2)}_{i,j_i}}_{L_2}^2\ ,
  \]
  where the last equality is the variance of a sum of independent,
  mean zero random variables. To continue, we need the following
  consequence of the Hoffman-J\o{rgensen} inequality
  (\cf~\cite[Proposition~6.10]{led-tal:91}). Here the notation
  $a\stackrel{\kappa}{\sim}b$ means that $a\leq \kappa\cdot b$ and
  $b\leq \kappa\cdot a$.
  \begin{thm}
    Given $0<p,q<\infty$, there is a constant $K_{p,q}$ such that if
    $\cX_1,\dots, \cX_N$ are independent, symmetric random variables in
    $L_p$ then
    \[
    \Bignorm{\sum_{i=1}^N \cX_i}_{L_p} \stackrel{K_{p,q}}{\sim}
    \Bignorm{\max_{1\leq i\leq N} \abs{\cX_i}}_{L_p} +
    \Bignorm{\sum_{i=1}^N \cX_i\cdot \bi_{\{\abs{\cX_i}\leq \delta_0\}}}_{L_q}
    \]
    where $\delta_0=\inf \Big\{ t>0:\,\sum_{i=1}^N
    \bigprob{\abs{\cX_i}>t}\leq \frac1{8\cdot 3^p}  \Big\}$.
  \end{thm}
  We apply this theorem to the sequence
  $\big(S^{(2)}_{i,j_i}\big)_{i=1}^m$ (where $j\in\cF$) with
  $p=2,\ q=1$ to obtain
  \begin{align*}
    K_{2,1}^{-1}\cdot \Bignorm{\sum_{i=1}^m S^{(2)}_{i,j_i}}_{L_2} &\leq
    \Bignorm{\max _{1\leq i\leq m}\abs{S^{(2)}_{i,j_i}}}_{L_2} +
    \Bignorm{\sum_{i=1}^m S^{(2)}_{i,j_i}\cdot
      \bi_{\abs{S^{(2)}_{i,j_i}}\leq \delta_0}}_{L_1}\\
    &\leq C + K\cdot \Bignorm{\sum_{i=1}^m T_{i,j_i}}_{L_1} \leq C+K\ .
  \end{align*}
  The second inequality follows by applying the square function
  inequality twice and monotonicity of expectation. Substituting this
  into the previous inequality, we obtain $\pi_2(S^{(2)}) \leq
  K_{2,1}\cdot (C+K)$.

  We have thus shown that $\pi_2(S)\leq\pi_2(S^{(1)})+\pi_2(S^{(2)})
  \leq K_G\cdot n + K_{2,1}\cdot (C+K)$. This upper bound is
  independent of $m$, and so the proof is complete.
\end{proof}

\def\cprime{$'$} \def\cprime{$'$} \def\cprime{$'$}

\vspace{2ex}

\noindent
\parbox[t]{0.5\textwidth}{%
N.~J.~Laustsen and A.~Zs\'ak\\
Department of Mathematics and Statistics\\
Lancaster University\\
Lancaster\\
LA1 4YF, United Kingdom\\
email: \texttt{n.laustsen@lancaster.ac.uk\\
   a.zsak@dpmms.cam.ac.uk}}%
\hfill%
\parbox[t]{0.5\textwidth}{%
E.~Odell\\
Department of Mathematics,\\
The University of Texas,\\
1 University Station C1200,\\
Austin, TX 78712, USA\\
email: \texttt{odell@math.utexas.edu}}

\vspace{2ex}

\noindent
\parbox[t]{0.5\textwidth}{%
A.~Zs\'ak also at:\\
Peterhouse\\
Cambridge\\
CB2 1RD, United Kingdom}
\hfill%
\parbox[t]{0.5\textwidth}{%
Th.~Schlumprecht\\
Department of Mathematics,\\
Texas A\&M University,\\
College Station, TX 78712, USA\\
email: \texttt{schlump@math.tamu.edu}}

\end{document}